\documentclass[12pt,onecolumn]{IEEEtran}

\usepackage{graphicx,color}
\usepackage{amsmath,amssymb,mathrsfs,theorem}
\usepackage[vlined,ruled]{algorithm2e}
\usepackage{subfigure}

\IEEEoverridecommandlockouts
\author{
  \parbox{2 in}{ \centering Fabio Pasqualetti \thanks{Fabio
      Pasqualetti is with the Center for Control, Dynamical Systems
      and Computation, University of California, Santa Barbara, {\tt
        fabiopas@engineering.ucsb.edu} .}  } \quad
  \parbox{2 in}{ \centering Antonio Bicchi
         \thanks{Antonio Bicchi is with the Centro I. R. ``E. Piaggio'',
             Universit\`{a} di Pisa,
             {\tt bicchi@ing.unipi.it}  .}
         }\quad
         \parbox{2 in}{ \centering Francesco Bullo            
           \thanks{Francesco Bullo is with the Center for Control,
             Dynamical Systems and Computation, University of California,
             Santa Barbara,
             {\tt bullo@engineering.ucsb.edu} .}
         }
}

\title{Consensus Computation in Unreliable Networks:\\
  A System Theoretic Approach\thanks{This material is based upon work
    supported in part by the ARO Institute for Collaborative
    Biotechnology award DAAD19-03-D-0004, by the AFOSR MURI award
    FA9550-07-1-0528, by the Contract IST 224428 (2008) (STREP) ``CHAT
    - Control of Heterogeneous Automation Systems: Technologies for
    scalability, reconfigurability and security,'' and by the CONET,
    the Cooperating Objects Network of Excellence, funded by the
    European Commission under FP7 with contract number
    FP7-2007-2-224053. The authors thank Dr.\ Natasha Neogi for
    insightful conversations, and the reviewers for their thoughtful
    and constructive remarks.}}

\usepackage{theorem}
\newtheorem{theorem}{Theorem}[section]

\newtheorem{definition}{Definition}
\newtheorem{lemma}{Lemma}[section]
{\theorembodyfont{\rmfamily} 
\newtheorem{remark}{Remark}

\newcommand\oprocendsymbol{\hbox{$\square$}}
\newcommand\oprocend{\relax\ifmmode\else\unskip\hfill\fi\oprocendsymbol}

\renewcommand{\natural}{{\mathbb{N}}}
\renewcommand{\int}{{\mathbb{Z}}}
\newcommand{\real}{{\mathbb{R}}}

\newcommand{\subscr}[2]{{#1}_{\textup{#2}}}
\newcommand{\until}[1]{\{1,\dots,#1\}}

\newcommand{\Ker}{\operatorname{Ker}}
\newcommand{\Image}{\operatorname{Im}}
\newcommand{\Vtar}{\mathcal{V}^*}
\newcommand{\Star}{\mathcal{S}^*}

\newcommand{\B}{\mathcal{B}}
\newcommand{\C}{\mathcal{C}}
\newcommand{\transpose}{\mathsf{T}} 

\renewcommand{\theenumi}{(\roman{enumi})}

\def\Algo1{\textit{Complete Identification}}

\renewcommand{\baselinestretch}{.99}

\begin{document}

\maketitle
\begin{abstract}
  This work addresses the problem of ensuring trustworthy computation
  in a linear consensus network. A solution to this problem is
  relevant for several tasks in multi-agent systems including motion
  coordination, clock synchronization, and cooperative estimation. In
  a linear consensus network, we allow for the presence of
  \emph{misbehaving agents}, whose behavior deviate from the nominal
  consensus evolution. We model misbehaviors as unknown and
  unmeasurable inputs affecting the network, and we cast the
  misbehavior detection and identification problem into an
  unknown-input system theoretic framework. We consider two extreme
  cases of misbehaving agents, namely \emph{faulty} (non-colluding)
  and \emph{malicious} (Byzantine) agents. First, we characterize the
  set of inputs that allow misbehaving agents to affect the consensus
  network while remaining undetected and/or unidentified from certain
  observing agents. Second, we provide worst-case bounds for the
  number of concurrent faulty or malicious agents that can be detected
  and identified. Precisely, the consensus network needs to be $2k+1$
  (resp. $k+1$) connected for $k$ malicious (resp. faulty) agents to
  be generically detectable and identifiable by every well behaving
  agent. Third, we quantify the effect of undetectable inputs on the
  final consensus value. Fourth, we design three algorithms to detect
  and identify misbehaving agents. The first and the second algorithm
  apply fault detection techniques, and affords complete detection and
  identification if global knowledge of the network is available to
  each agent, at a high computational cost. The third algorithm is
  designed to exploit the presence in the network of weakly
  interconnected subparts, and provides local detection and
  identification of misbehaving agents whose behavior deviates more
  than a threshold, which is quantified in terms of the
  interconnection structure.
\end{abstract}

\section{Introduction}
Distributed systems and networks have received much attention in the
last years because of their flexibility and computational
performance. One of the most frequent tasks to be accomplished by
autonomous agents is to agree upon some parameters. Agreement
variables represent quantities of interest such as the work load in a
network of parallel computers, the clock speed for wireless sensor
networks, the velocity, the rendezvous point, or the formation pattern
for a team of autonomous vehicles; e.g., see
\cite{WR-RWB-EMA:07,ROS-JAF-RMM:07,FB-JC-SM:09}.

Several algorithms achieving consensus have been proposed and studied
in the computer science community \cite{NAL:97}. In this work, we
consider linear consensus iterations, where, at each time instant,
each node updates its state as a weighted combination of its own value
and those received from its neighbors
\cite{AJ-JL-ASM:02,ROS-RMM:03c}. The choice of algorithm weights
influences the convergence speed toward the steady state value
\cite{LX-SB:04}.

Because of the lack of a centralized entity that monitors the activity
of the nodes of the network, distributed systems are prone to attacks
and components failure, and it is of increasing importance to
guarantee trustworthy computation even in the presence of misbehaving
parts \cite{PM-RM:02}. Misbehaving agents can interfere with the
nominal functions of the network in different ways. In this paper, we
consider two extreme cases: that the deviations from their nominal
behavior are due to genuine, random faults in the agents; or that
agents can instead craft messages with the purpose of disrupting the
network functions.  In the first scenario, faulty agents are unaware
of the structure and state of the network and ignore the presence of
other faults. In the second scenario, the worst-case assumption is
made that misbehaving agents have knowledge of the structure and state
of the network, and may collude with others to produce the biggest
damage. We refer to the first case as non-colluding, or faulty; to the
second case as malicious, or Byzantine.

Reaching unanimity in an unreliable system is an important problem,
well studied by computer scientists interested in distributed
computing. A first characterization of the resilience of distributed
systems to malicious attacks appears in \cite{LL-RSMP:82}, where the
authors consider the task of agreeing upon a binary message sent by a
``Byzantine general,'' when the communication graph is complete. In \cite{DD:82}
the resilience of a partially connected\footnote{The connectivity of a
  graph is the maximum number of disjoint paths between any two
  vertices of the graph. A graph is complete if it has connectivity
  $n-1$, where $n$ is the number of vertices in the graph.} network
seeking consensus is analyzed, and it is shown that the well-behaving
agents of a network can always agree upon a parameter if and only if
the number of malicious agents
\begin{enumerate} 
\item is less than $1 /2$ of the network connectivity, and 
\item is less than $1 /3$ of the number of processors. 
\end{enumerate} 
This result has to be regarded as a fundamental limitation of the
ability of a distributed consensus system to sustain arbitrary
malfunctioning: the presence of misbehaving Byzantine processors can
be tolerated only if their number satisfies the above threshold,
independently of whatever consensus protocol is adopted.

We consider linear consensus algorithms in which every agent,
including the misbehaving ones, are assumed to send the same
information to all their neighbors. This assumption appears to be
realistic for most control scenarios. In a sensing network for
instance, the data used in the consensus protocol consist of the
measurements taken directly by the agents, and (noiseless)
measurements regarding the same quantity coincide. Also, in a
broadcast network, the information is transmitted using broadcast
messages, so that the content of a message is the same for all the
receiving nodes. The problem of characterizing the resilience
properties of linear consensus strategies has been partially addressed
in recent works \cite{FP-AB-FB:06v,SS-CNH:08fault,SS-CNH:08bis},
where, for the malicious case, it is shown that, despite the limited
abilities of the misbehaving agents, the resilience to external
attacks is still limited by the connectivity of the network. In
\cite{FP-AB-FB:06v} the problem of detecting and identifying
misbehaving agents in a linear consensus network is first introduced,
and a solution is proposed for the single faulty agent case. In
\cite{SS-CNH:08fault,SS-CNH:08bis}, the authors provide one policy
that $k$ malicious agents can follow to prevent some of the nodes of a
$2k$-connected network from computing the desired function of the
initial state, or, equivalently, from reaching an agreement. On the
contrary, if the connectivity is $2k+1$ or more, then the authors show
that generically the set of misbehaving nodes is identified
independent of its behavior, so that the desired consensus is
eventually reached.

The main differences between this paper and the references
\cite{SS-CNH:08fault,SS-CNH:08bis} are as follows. First, the method
proposed in \cite{SS-CNH:08fault,SS-CNH:08bis} takes inspiration from
parity space methods for fault detection, while, following our early
work \cite{FP-AB-FB:06v}, we adopt here unknown-input observers
techniques \cite{RP-PF-RC:89}. Second, we focus on consensus networks,
and we derive specific results for this important case that cannot be
assessed for general linear iterations. Third, we consider two
different types of misbehaving agents, namely malicious and faulty
agents, and we provide network resilience bounds for both
cases. Fourth, we exhaustively characterize the complete set of
policies that make a set of $k$ agents undetectable and/or
unidentifiable, as opposed to \cite{SS-CNH:08fault} where only a
particular disrupting strategy is defined. Fifth, we study system
theoretic properties of consensus systems (e.g., detectability,
stabilizability, left-invertibility), and we quantify the effect of
some misbehaving inputs on the network performance. Finally, we
address the problem of detection complexity and we propose a
computationally efficient detection method, as opposed to
combinatorial procedures. Our approach also differs from the existing
computer science literature, e.g., our analysis leads to the
development of algorithms that can be easily extended to work on both
discrete and continuous time linear consensus networks, and also with
partial knowledge of the network topology.

The main contributions of this work are as follows. By recasting the
problem of linear consensus computation in an unreliable system into a
system theoretic framework, we provide alternative and constructive
system-theoretic proofs of existing bounds on the number of
identifiable misbehaving agents in a linear network, i.e., $k$
Byzantine agents can be detected and identified if the network is
$(2k+1)$-connected, and they cannot be identified if the network is
$2k$-connected or less. Moreover, by showing some connections between
linear consensus networks and linear dynamical systems, we
exhaustively describe the strategies that misbehaving nodes can follow
to disrupt a linear network that is not sufficiently connected. In
particular, we prove that the inputs that allow the misbehaving agents
to remain undetected or unidentified coincide with the inputs-zero of
a linear system associated with the consensus network.
We provide a novel and comprehensive analysis on the detection and
identification of non-colluding agents. We show that $k$ faulty agents
can be identified if the network is $(k+1)$-connected, and cannot if
the network is $k$-connected or less. For both the cases of Byzantine
and non-colluding agents, we prove that the proposed bounds are
generic with respect to the network communication weights, i.e., given
an (unweighted) consensus graph, the bounds hold for almost all
(consensus) choices of the communication weights. In other words, if
we are given a $(k+1)$-connected consensus network for which $k$
faulty agents cannot be identified, then a random and arbitrary small
change of the communication weights (within the space of consensus
weights) make the misbehaving agents identifiable with probability
one. In the last part of the paper, we discuss the problem of
detecting and identifying misbehaving agents when either the partial
knowledge of the network or hardware limitations make it impossible to
implement an exact identification procedure. We introduce a notion of
network decentralization in terms of relatively weakly connected
subnetworks. We derive a sufficient condition on the consensus matrix
that allows to identify a certain class of misbehaving agents under
local network model information. Finally, we describe a local algorithm
to promptly detect and identify corrupted components.

The rest of the paper is organized as follows. Section \ref{notation}
briefly recalls some basic facts on the geometric approach to the
study of linear systems, and on the fault detection and isolation
problem. In Section \ref{linear_iterations} we model linear consensus
networks with misbehaving agents. Section \ref{resilience_section}
presents the conditions under which the misbehaving agents are
detectable and identifiable. In Section
\ref{undetectble_attack_section} we characterize the effect of an
unidentifiable attack on the network consensus state. In Section
\ref{generic_section} we show that the resilience of linear consensus
networks to failures and external attacks is a generic property with
respect to the consensus weights. In Section \ref{approximate} we
present our algorithmic procedures. Precisely we derive an exact
identification algorithm, and an approximate and low-complexity
procedure. Finally, Sections \ref{examples} and \ref{conclusions}
contain respectively our numerical studies and our conclusion.

\section{Notation and preliminary concepts}\label{notation}
We adopt the same notation as in \cite{GB-GM:91}. Let $n,m,p \in
\natural$, let $A \in \real^{n \times n}$, $B \in \real^{n \times m}$,
and $C \in \real^{p \times n}$. Let the triple $(A,B,C)$ denote the
linear discrete time system
\begin{equation}\label{lds}\begin{split}
  x(t+1)&=Ax(t)+Bu(t),\\
  y(t)&=Cx(t),
\end{split}\end{equation}
and let the subspaces $\B \subseteq \real^{n \times n}$ and $\C
\subseteq \real^{n \times n}$ denote the image space $\Image(B)$ and
the null space $\Ker(C)$, respectively. A subspace
$\mathcal{V}\subseteq \real^{n \times n}$ is a $(A,\B)$-controlled
invariant if $A\mathcal{V}\subseteq \mathcal{V}+\mathcal{B}$, while a
subspace $\mathcal{S}\subseteq \real^{n \times n}$ is a
$(A,\C)$-conditioned invariant if $A(\mathcal{S}\cap \C)\subseteq
\mathcal{S}$.
The set of all controlled invariants contained in $\C$ admits a
supremum, which we denote with $\Vtar$, and which corresponds to the
locus of all possible state trajectories of \eqref{lds} invisible at
the output.
On the other hand, the set of the conditioned invariants containing
$\B$ admits an infimum, which we denote with $\Star$. Several
problems, including disturbance decoupling, non interacting control,
fault detection and isolation, and state estimation in the presence of
unknown inputs have been addressed and solved in a geometric framework
\cite{GB-GM:91,WMW:85}.

In the classical Fault Detection and Isolation (FDI) setup, the
presence of sensor failures and actuator malfunctions is modeled by
adding some unknown and unmeasurable functions $u_i (t)$ to the
nominal system. The FDI problem is to design, for each failure $i$, a
filter of the form
\begin{equation}\label{res_gen}\begin{split}
  w_i(t+1)&=F_iw_i(t)+E_iy(t),\\
  r_i(t)&=M_iw(t)+H_iy(t),
\end{split}\end{equation}
also known as residual generator, that takes the observables $y(t)$
and generates a residual vector $r_i(t)$ that allows to uniquely
identify if $u_i (t)$ becomes nonzero, i.e., if the failure $i$
occurred in the system. Let $B_1,\dots,B_m$ be the input matrices of
the failure functions $u_1,\dots,u_m$. As a result of
\cite{GB-GM:91,MAM-GCV-ASW-CM:89}, the $i$-th failure can be correctly
identified if and only if $\mathcal{B}_i\cap
(\mathcal{V}^{*}_{K\setminus \{i\}}+\mathcal{S}^{*}_{K\setminus
  \{i\}})=\emptyset$, where $\Vtar_{K\setminus \{i\}}$ and
$\Star_{K\setminus \{i\}}$ are the maximal controlled and minimal
conditioned invariant subspaces associated with the triple $(A,[B_{1}
\cdots B_{i-1} \enspace B_{i+1} \cdots B_{m}],C)$. It can be shown
that, under the above solvability condition, the filter
\eqref{res_gen} can be designed as a dead beat device to have finite
convergence time \cite{MAM-GCV-ASW-CM:89}: this property will be used
in Section \ref{approximate} for the characterization of our intrusion
detection algorithm. We remark that, although the FDI problem does not
coincide with the problem we are going to face, we will be using some
standard FDI techniques to design our detection and identification
algorithms, and we refer the reader to \cite{RP-PF-RC:89} for a
comprehensive treatment of the subject.

\section{Linear consensus in the presence of misbehaving
  agents}\label{linear_iterations}
Let $G$ denote a directed graph with vertex set $V=\until{n}$ and edge
set $E \subset V \times V$, and recall that the connectivity of $G$ is
the maximum number of disjoint paths between any two vertices of the
graph, or, equivalently, the minimum number of vertices in a vertex
cutset \cite{CDG-GFR:01}. The neighbor set of a node $i\in V$, i.e.,
all the nodes $j \in V$ such that the pair $(j,i)\in E$, is denoted
with $N_{i}$. We let each vertex $j\in V$ denote an autonomous agent,
and we associate a real number $x_j$ with each agent $j$. Let the
vector $x \in \real^{n}$ contain the values $x_j$. A linear iteration
over $G$ is an update rule for $x$ and is described by the linear
discrete time system
\begin{align}\label{linearsystem}
  x(t+1)=Ax(t),
\end{align}
where the $(i,j)$-th entry of $A$ is nonzero only if $(j,i) \in E$. If
the matrix $A$ is row stochastic and primitive, then, independent of
the initial values of the nodes, the network asymptotically converges
to a configuration in which the state of the agents coincides. In the
latter case, the matrix $A$ is referred to as a \emph{consensus
  matrix}, and the system \eqref{linearsystem} is called
\textit{consensus system}. The graph $G$ is referred to as the
communication graph associated with the consensus system
\eqref{linearsystem} or, equivalently, with the consensus matrix
$A$. A detailed treatment of the applications, and the convergence
aspects of the consensus algorithm is in
\cite{WR-RWB-EMA:07,ROS-JAF-RMM:07,FB-JC-SM:09}, and in the references
therein.

We allow for some agents to update their state differently than
specified by the matrix $A$ by adding an exogenous input to the
consensus system. Let $u_i (t)$, $i\in V$, be the input associated
with the $i$-th agent, and let $u (t)$ be the vector of the functions
$u_i (t)$. The consensus system becomes $x(t+1)=Ax(t)+u(t)$.
\begin{definition}[Misbehaving agent]
  An agent $j$ is \emph{misbehaving} if there exists a time $t\in
  \mathbb{N}$ such that $u_j(t)\neq 0$.
\end{definition}
In Section \ref{resilience_section} we will give a precise definition
of the distinction, made already in the Introduction, between {\em
  faulty} and {\em malicious} agents on the basis of their inputs.

Let $K=\{i_{1},i_{2},\dots\}\subseteq V$ denote a set of misbehaving
agents, and let $B_K=[e_{i_{1}} \enspace e_{i_{2}} \enspace \cdots ]$,
where $e_i$ is the $i$-th vector of the canonical basis. The consensus
system with misbehaving agents $K$ reads as
\begin{align}\label{consensus_intruders}
  x(t+1)=Ax(t)+B_Ku_K(t).
\end{align}
As it is shown in \cite{FP-AB-FB:06v}, algorithms of the form
(\ref{linearsystem}) have no resilience to malfunctions, and the
presence of a misbehaving agent may prevent the entire network from
reaching consensus. As an example, let $c\in \mathbb{R}$, and let $u_i
(t)=-A_ix(t)+c$, being $A_i$ the $i$-th row of $A$. After reordering
the variables in a way that the well-behaving nodes come first, the
consensus system can be rewritten as
\begin{align}\label{absorbing}
  \tilde{x}(t+1)=
  \begin{bmatrix}
      Q & R\\
      0 & 1
  \end{bmatrix}
  \tilde{x}(t),
\end{align}
where the matrix $Q$ corresponds to the interaction among the nodes
$V\setminus \{i\}$, while $R$ denotes the connection between the sets
$V\setminus \{i\}$ and $\{i\}$. Recall that a matrix is said to be Schur
stable if all its eigenvalues lie in the open unit disk.
\begin{lemma}[Quasi-stochastic submatrices]\label{quasi_stoc}
  Let $A$ be an $n\times n$ consensus matrix, and let $J$ be a proper
  subset of $\{1,\dots,n\}$. The submatrix with entries
  $A_{i,k}$, $i, k \in J$, is Schur stable.
\end{lemma}
\begin{proof}
  Reorder the nodes such that the indexes in $J$ come first in the
  matrix $A$. Let $A_J$ be the leading principal submatrix of
  dimension $|J|$. Let
$\tilde{A}_J=\left[
    \begin{smallmatrix}
      A_J & 0\\
      0 & 0
    \end{smallmatrix}
    \right]$,
    where the zeros are such that $\tilde{A}_J$ is $n\times n$, and
    note that $\rho(A_J)=\rho(\tilde{A}_J)$, where $\rho(A_J)$ denotes
    the spectral radius of the matrix $A_J$ \cite{CDM:01}. Since $A$
    is a consensus matrix, it has only one eigenvalue of unitary
    modulus, and $\rho(A)=1$. Moreover, $A \ge |\tilde{A}_J|$, and $A
    \neq |\tilde{A}_J|$, where $|\tilde{A}_J|$ is such that its
    $(i,j)$-th entry equals the absolute value of the $(i,j)$-th entry
    of $\tilde{A}_J$, $\forall i,j$. It is known that $\rho(A_J)\le
    \rho(A)=1$, and that if equality holds, then there exists a
    diagonal matrix $D$ with nonzero diagonal entries, such that
    $A=D\tilde{A}_JD^{-1}$ \cite[Wielandt's Theorem]{CDM:01}. Because
    $A$ is irreducible, there exists no diagonal $D$ with nonzero
    diagonal entries such that $A=D\tilde{A}_JD^{-1}$ and the
    statement follows.
\end{proof}

Because of Lemma \ref{quasi_stoc}, the matrix $Q$ in \eqref{absorbing}
is Schur stable, so that the steady state value of the well-behaving
agents in \eqref{absorbing} depends upon the action of the misbehaving
node, and it corresponds to $(I-Q)^{-1}Rc$. In particular, since
$(I-Q)^{-1}R = [1 \, \cdots \, 1]^\transpose$, a single misbehaving agent can
steer the network towards any consensus value by choosing the constant
$c$.\footnote{If the misbehaving input is not constant, then the
  network may not achieve consensus. In particular, the effect of a
  misbehaving input $u_K$ on the network state at time $t$ is given by
  $\sum_{\tau =0}^{t} A^{t-\tau} B_K u_K (\tau)$ (see also Section
\ref{undetectble_attack_section}).}

It should be noticed that a different model for the misbehaving nodes
consists in the modification of the entries of $A$ corresponding to
their incoming communication edges. However, since the resulting
network evolution can be obtained by properly choosing the input $u_K
(t)$ and letting the matrix $A$ fixed, our model does not limit
generality, while being convenient for the analysis. For the same
reason, system \eqref{consensus_intruders} also models the case of
defective communication edges. Indeed, if the edge from the node $i$
to the node $j$ is defective, then the message received by the agent
$j$ at time $t$ is incorrect, and hence also the state $x_j(\bar t)$,
$\bar t\ge t$. Since the values $x_j(\bar t)$ can be produced with an
input $u_j (t)$, the failure of the edge $(i,j)$ can be regarded as
the $j$-th misbehaving action. Finally, the following key difference
between our model and the setup in \cite{DD:82} should be noticed. If
the communication graph is complete, then up to $n-1$ (instead of
$\lfloor n/3 \rfloor$) misbehaving agents can be identified in our
model by a well-behaving agent. Indeed, since with a complete
communication graph the initial state $x(0)$ is correctly received by
every node, the consensus value is computed after one communication
round, so that the misbehaving agents cannot influence the dynamics of
the network.


\section{Detection and identification of misbehaving
  agents}\label{resilience_section}
The problem of ensuring trustworthy computation among the agents of a
network can be divided into a detection phase, in which the presence
of misbehaving agents is revealed, and an identification phase, in
which the identity of the intruders is discovered. A set of
misbehaving agents may remain undetected from the observations of a
node $j$ if there exists a normal operating condition under which the
node would receive the same information as under the perturbation due
to the misbehavior. To be more precise, let $C_j=[e_{n_1} \enspace
\dots \enspace e_{n_p}]^\transpose$, $\{n_1,\dots,n_p\}= N_j$, denote the
output matrix associated with the agent $j$, and let $y_j (t) = C_j
x(t)$ denote the measurements vector of the $j$-th agent at time
$t$. Let $x(x_0,\bar u,t)$ denote the network state trajectory
generated from the initial state $x_0$ under the input sequence $\bar
u(t)$, and let $y_j(x_0,\bar u,t)$ be the sequence measured by the
$j$-th node and corresponding to the same initial condition and input.

\begin{definition}[Undetectable input]\label{undetectable_input}
  For a linear consensus system of the form
  \eqref{consensus_intruders}, the input $u_{K}(t)$ introduced by a
  set $K$ of misbehaving agents is \emph{undetectable} if
  \begin{align*}
    \exists \, x_1 , x_2 \in \real^n, j \in V : \forall t\in
    \mathbb{N}, y_j(x_1,u_K,t) = y_j(x_2,0,t).
  \end{align*}
\end{definition}

A more general concern than detection is identifiability of intruders,
i.e. the possibility to distinguish from measurements between the
misbehaviors of two distinct agents, or, more generally, between two
disjoint subsets of agents. Let $\mathcal{K} \subset 2^V$ contain all
possible sets of misbehaving agents.\footnote{An element of
  $\mathcal{K}$ is a subset of $\until{n}$. For instance,
  $\mathcal{K}$ may contain all the subsets of $\until{n}$ with a
  specific cardinality.}

\begin{definition}[Unidentifiable input]\label{unidentifiable_input}
  For a linear consensus system of the form
  \eqref{consensus_intruders} and a nonempty set $K_1 \in
  \mathcal{K}$, an input $u_{K_1}(t)$ is \emph{unidentifiable} if
  there exist $K_2 \in \mathcal{K}$, with $K_1 \neq K_2$, and an input
  $u_{K_2}(t)$ such that
  \begin{align*}
    \exists \, x_1,x_2 \!\in \real^n, j\in V \!:\!  \forall t\in
    \mathbb{N}, y_j(x_1,u_{K_1},t) = y_j(x_2,u_{K_2},t).
  \end{align*}
\end{definition}
Of course, an undetectable input is also unidentifiable, since it
cannot be distinguished from the zero input. The converse does not
hold. Unidentifiable inputs are a very specific class of inputs, to be
precisely characterized later in this section. Correspondingly, we
define

\begin{definition}[Malicious behaviors]\label{malicious_definition}
  A set of misbehaving agents $K$ is malicious if its input $u_K(t)$
  is unidentifiable. It is faulty otherwise.
\end{definition}

We provide now a characterization of malicious behaviors for the
particularly important class of linear consensus networks. 
Notice however that, if the matrix $A$ below is not restricted to be a
consensus matrix, then the following Theorem extends the results in
\cite{SS-CNH:08fault} by fully characterizing the inputs for which a
group of misbehaving agents remains unidentified from the output
observations of a certain node.

\begin{theorem}[Characterization of malicious
  behaviors]\label{behaviors}
  For a linear consensus system of the form
  \eqref{consensus_intruders} and a nonempty set $K_1 \in
  \mathcal{K}$, an input $u_{K_1} (t)$ is unidentifiable if and only
  if
  \begin{align*}
    C_j A^{t+1} \bar x = \sum_{\tau =0}^{t} C_j A^{t-\tau} \left (
      B_{K_1} u_{K_1} (\tau) - B_{K_2} u_{K_2} (\tau) \right),
  \end{align*}
  for all $t \in \mathbb{N}$, and for some $u_{K_2}(t)$, with $K_2 \in
  \mathcal{K}$, $K_1 \neq K_2$, and $\bar x \in \real^n$. If the
  same holds with $u_{K_2}(t) \equiv 0$, the input is actually
  undetectable.
\end{theorem}
\begin{proof}
  By definitions \ref{undetectable_input} and
  \ref{unidentifiable_input}, an input $u_{K_1} (t)$ is unidentifiable
  if $y_j (x_1, u_{K_{1}},t) = y_j (x_2, u_{K_{2}},t)$, and it is
  undetectable if $y_j (x_1, u_{K_{1}},t) = y_j (x_2,0,t)$, for some
  $x_1$, $x_2$, and $u_{K_{2}} (t)$. Due to linearity of the network,
  the statement follows.
\end{proof}


\begin{remark}[Malicious behaviors are not generic]
  Because an unidentifiable input must satisfy the equation in
  Theorem~\ref{behaviors}, excluding pathological cases,
  unidentifiable signals are not generic, and they can be injected
  only intentionally by colluding misbehaving agents. This motivates
  our definition of ``malicious'' for those agents which use
  unidentifiable inputs.
  \oprocend
\end{remark}




We consider now the resilience of a consensus network to faulty and
malicious misbehaviors. Let $I$ denote the identity matrix of
appropriate dimensions. The zero dynamics of the linear system
$(A,B_K,C_j)$ are the (nontrivial) state trajectories invisible at the
output, and can be characterized by means of the $(n+p) \times (n+ |K|)$
pencil
\begin{align*}
  P(z)=
  \begin{bmatrix}
    zI-A & B_K\\
    C_j & 0
  \end{bmatrix}.
\end{align*}
The complex value $\bar z$ is said to be an invariant zero of the
system $(A,B_K,C_j)$ if there exists a state-zero direction $x_0$,
$x_0 \neq 0$, and an input-zero direction $g$, such that $(\bar z
I-A)x_0+B_K g=0$, and $C_jx_0=0$. Also, if $\textup{rank}(P(z))=n+|K|$
for all but finitely many complex values $z$, then the system
$(A,B_K,C_j)$ is left-invertible, i.e., starting from any initial
condition, there are no two distinct inputs that give rise to the same
output sequence \cite{HLT-AS-MH:01}. We next characterize the
relationship between the zero dynamics of a consensus system and the
connectivity of the consensus graph.

\begin{lemma}[Zero dynamics and connectivity]\label{left_connectivity}
  Given a $k$-connected linear network with matrix $A$, there exists a
  set of agents $K_1$, with $|K_1|>k$, and a node $j$ such that the
  consensus system $(A,B_{K_1},C_j)$ is not
  left-invertible. Furthermore, there exists a set of agents $K_2$,
  with $|K_2| = k$, and a node $j$ such that the system
  $(A,B_{K_2},C_j)$ has nontrivial zero dynamics.
\end{lemma}
\begin{proof}
  Let $G$ be the digraph associated with $A$, and let $k$ be the
  connectivity of $G$. Take a set $K$ of $k+1$ misbehaving nodes, such
  that $k$ of them form a vertex cut $S$ of $G$. Note that, since the
  connectivity of $G$ is $k$, such a set always exists. The network
  $G$ is divided into two subnetworks $G_1$ and $G_3$, which
  communicate only through the nodes $S$. Assume that the misbehaving
  agent $K\setminus S$ belongs to $G_3$, while the observing node $j$
  belongs to $G_1$. After reordering the nodes such that the vertices
  of $G_1$ come first, the vertices $S$ come second, and the vertices
  of $G_3$ come third, the consensus matrix $A$ is of the form
$\left[
    \begin{smallmatrix}
      A_{11} & A_{12} & 0\\
      A_{21} & A_{22} & A_{23}\\
      0 & A_{32} & A_{33}
    \end{smallmatrix}
  \right]$,
  where the zero matrices are due to the fact that $S$ is a vertex
  cut. Let $u_S (t)=-A_{23}x_3 (t)$, where $x_3$ is the vector
  containing the values of the nodes of $G_3$, and let $u_{K\setminus
    S} (t)$ be any arbitrary nonzero function. Clearly, starting from
  the zero state, the values of the nodes of $G_1$ are constantly $0$,
  while the subnetwork $G_3$ is driven by the misbehaving agent
  $K\setminus S$. We conclude that the triple $(A,B_K,C_j)$ is not
  left-invertible.

  Suppose now that $K\equiv S$ as previously defined, and let $u_K
  (t)=-A_{23}x_3 (t)$. Let the initial condition of the nodes of $G_1$
  and of $S$ be zero. Since every state trajectory generated by
  $x_3\neq 0$ does not appear in the output of the agent $j$, the
  triple $(A,B_K,C_j)$ has nontrivial zero dynamics.
\end{proof}


Following Lemma \ref{left_connectivity}, we next
state an upper bound on the number of misbehaving agents that can be
detected.

\begin{theorem}[Detection bound]\label{impossibility}
  Given a $k$-connected linear consensus network, there exist
  undetectable inputs for a specific set of $k$ misbehaving agents.
\end{theorem}
\begin{proof}
  Let $K$ , with $|K| = k$, be the misbehaving set, and let $K$ form a
  vertex cut of the consensus network. Because of Lemma
  \ref{left_connectivity}, for some output matrix $C_j$, the consensus
  system has nontrivial zero dynamics, i.e., there exists an initial
  condition $x(0)$ and an input $u_K (t)$ such that $y_j(t) = 0$ at
  all times. Hence, the input $u_K (t)$ is undetectable from the
  observations of $j$.
\end{proof}


We now consider the identification problem.

\begin{theorem}[Identification of misbehaving
  agents]\label{identification}
  For a set of misbehaving agents $K_1 \in \mathcal{K}$, every input
  is identifiable from $j$ if and only if the consensus system
  $(A,[B_{K_1}\enspace B_{K_2}],C_j)$ has no zero dynamics
  for every $K_2 \in \mathcal{K}$.
\end{theorem}
\begin{proof}
  \textit{(Only if)} By contradiction, let $x_0$ and
  $[u_{K_1}^\transpose \enspace -u_{K_2}^\transpose]^\transpose$ be a
  state-zero direction, and an input-zero sequence for the system
  $(A,[B_{K_1} \enspace B_{K_2}],C_j)$. We have
  \begin{align*}
    y_j(t)  = 0
    = C_j \!\bigg( \! A^tx_0 + \!\sum_{\tau =0}^{t-1}A^{t-\tau-1}
    \big( B_{K_1}u_{K_1} (\tau) - B_{K_2}u_{K_2} (\tau) \big) \!\bigg) .
  \end{align*}
    Therefore,
    \begin{align*}
      C_j\bigg(A^tx_0^1 +\sum_{\tau
        =0}^{t-1}A^{t-\tau-1}B_{K_1}u_{K_1} (\tau)\bigg)=C_j\bigg(A^tx_0^2
      +\sum_{\tau =0}^{t-1}A^{t-\tau-1}B_{K_2}u_{K_2} (\tau)\bigg),
    \end{align*}
    where $x_0^1-x_0^2=x_0$. Clearly, since the output sequence
    generated by $K_1$ coincide with the output sequence generated by
    $K_2$, the two inputs are unidentifiable.

    \textit{(If)} Suppose that, for any $K_2 \in \mathcal{K}$, the
    system $(A[B_{K_1} \; B_{K_2}])$ has no zero dynamics, i.e., there
    exists no initial condition $x_0$ and input $[u_{K_1}^\transpose
    \, u_{K_2}^\transpose]^\transpose$ that result in the output being
    zero at all times. By the linearity of the network, every input
    $u_{K_1}$ is identifiable.
\end{proof}

As a consequence of Theorem \ref{identification}, if up to $k$
misbehaving agents are allowed to act in the network, then a necessary
and sufficient condition to correctly identify the set of misbehaving
nodes is that the consensus system subject to any set of $2k$ inputs
has no nontrivial zero dynamics.
\begin{theorem}[Identification bound]\label{upper_malicious}
  Given a $k$-connected linear consensus network, there exist
  unidentifiable inputs for a specific set of $\lfloor
  \frac{k-1}{2}\rfloor + 1$ misbehaving agents.
\end{theorem}
\begin{proof}
  Since $2(\lfloor \frac{k-1}{2}\rfloor+1)\ge k$, by Lemma
  \ref{left_connectivity} there exist $K_1$, $K_2$, with $|K_1| =
  |K_2| = \lfloor \frac{k-1}{2}\rfloor+1$, and $j$ such that the
  system $(A,[B_{K_1} \enspace B_{K_2}],C_j)$ has nontrivial zero
  dynamics. By Theorem \ref{identification}, there exists an input and
  an initial condition such that $K_1$ is undistinguishable from $K_2$
  to the agent $j$.
\end{proof}


In other words, in a $k$-connected network, at most $k-1$
(resp. $\lfloor \frac{k-1}{2}\rfloor$) misbehaving agents can be
certainly detected (resp. identified) by every agent. Notice that, for
a linear consensus network, Theorem \ref{upper_malicious} provides an
alternative proof of the resilience bound presented in \cite{DD:82}
and in \cite{SS-CNH:08fault}.


We now focus on the faulty misbehavior case. Notice that, because such
agents inject only identifiable inputs by definition, we only need to
guarantee the existence of such inputs. We start by showing that,
independent of the cardinality of a set $K$, there exist detectable
inputs for a consensus system $(A,B_K,C_j)$, so that any set of faulty
agents is detectable. By using a result from \cite{JT:06}, an input
$u_K (t)$ is undetectable from the measurements of the $j$-th agent
only if for all $t \in \mathbb{N}$, it holds $C_j A^v B_K u_K (t) =
C_j A^{v+1} x(t)$, where $C_j A^v B_K$ is the first nonzero Markov
parameter, and $x(t)$ is the network state at time $t$. Notice that,
because of the irreducibility assumption of a consensus matrix,
independently of the cardinality of the faulty set and of the
observing node $j$, there exists a finite $v$ such that $C_j A^v B_K
\neq 0$, so that every input $u_K (t) \neq (C_j A^v B_K )^\dag C_j
A^{v+1} x(t)$ is detectable. We show that, if the number of
misbehaving components is allowed to equal the connectivity of the
consensus network, then there exists a set of misbehaving agents that
are unidentifiable independent of their input.

\begin{theorem}[Identification of faulty agents]\label{upper}
  Given a $k$-connected linear consensus network, there exists no
  identifiable input for a specific set of $k$ misbehaving agents
\end{theorem}
\begin{proof}
    Let $K_1$, with $|K_1| = k$, form a vertex cut. The network is
  divided into two subnetworks $G_1$ and $G_2$ by the agents
  $K_1$. Let $K_2$, with $|K_2| \le k$, be the set of faulty agents,
  and suppose that the set $K_2$ belongs to the subnetwork $G_2$. Let
  $j$ be an agent of $G_1$. Notice that, because $K_1$ forms a vertex
  cut, for every initial condition $x(0)$ and for every input $u_{K_2}
  (t)$, there exists an input $u_{K_1} (t)$ such that the output
  sequences at the node $j$ coincide. In other words, every input
  $u_{K_2} (t)$ is unidentifiable.
\end{proof}


Hence, in a $k$-connected network, a set of $k$ faulty agents may
remain unidentified independent of its input function.
It should be noticed that Theorems \ref{upper_malicious} and
\ref{upper} only give an upper bound on the maximum number of
concurrent misbehaving agents that can be detected and identified. In
Section~\ref{generic_section} it will be shown that, generically, in a
$k$-connected network, there exists only identifiable inputs for any
set of $\lfloor \frac{k-1}{2} \rfloor$ misbehaving agents, and that
there exist some identifiable inputs for any set of $k-1$ misbehaving
agents. In other words, if there exists a set of misbehaving nodes
that cannot be identified by an agent, then, provided that the
connectivity of the communication graph is sufficiently high, a random
and arbitrarily small change of the consensus matrix makes the
misbehaving nodes detectable and identifiable with probability one.

\section{Effects of unidentified misbehaving
  agents}\label{undetectble_attack_section}
In the previous section, the importance of zero dynamics in the
misbehavior detection and identification problem has been shown. In
particular, we proved that a misbehaving agent may alter the nominal
network behavior while remaining undetected by injecting an input-zero
associated with the current network state. We now study the effect of
an unidentifiable attack on the final consensus value. As a
preliminary result, we prove the detectability of a consensus network.

\begin{lemma}[Detectability]\label{detectability}
  Let the matrix $A$ be row stochastic and irreducible. For any
  network node $j$, the pair $(A,C_j)$ is detectable.
\end{lemma}
\begin{proof}
  If $A$ is stochastic and irreducible, then it has at least $h\geq 1$
  eigenvalues of unitary modulus. Precisely, the spectrum of $A$
  contains
  $\{1=e^{i\theta_0},e^{i\theta_1},\dots,e^{i\theta_{h-1}}\}$. By
  Wielandt's theorem \cite{CDM:01}, we have $AD_k=e^{i\theta_k}D_k A$,
  where $k\in \{0,\dots,h-1\}$, and $D_k$ is a full rank diagonal
  matrix. By multiplying both sides of the equality by the vector of
  all ones, we have
  \mbox{$AD_k\textbf{1}=e^{i\theta_k}D_kA\textbf{1}=e^{i\theta_k}D_k\textbf{1}$},
  so that $D_k\textbf{1}$ is the eigenvector associated with the
  eigenvalue $e^{i\theta_k}$. Observe that the vector $D_k\textbf{1}$
  has no zero component, and that, by the eigenvector test
  \cite{HLT-AS-MH:01}, the pair $(A,C_j)$ is detectable. Indeed, since
  $A$ is irreducible, the neighbor set $N_j$ is nonempty, and the
  eigenvector $D_k \textbf{1}$, with $k\in \{0,\dots,h-1\}$, is not
  contained in $\Ker(C_j)$.
\end{proof}

Observe that the primitivity of the network matrix is not assumed
Lemma \ref{detectability}. By duality, a result on the stabilizability
of the pair $(A,B_j)$ can also be asserted.
\begin{lemma}[Stabilizability]
  Let the matrix $A$ be row stochastic and
  irreducible. For any network node $j$, the pair $(A,B_j)$ is
  stabilizable.
\end{lemma}

\begin{remark}[State estimation via local computation]
  If a linear system is detectable (resp. stabilizable), then a linear
  observer (resp. controller) exists to asymptotically estimate
  (resp. stabilize) the system state. By combining the above results
  with Lemma \ref{quasi_stoc}, we have that, under a mild assumption
  on the matrix $A$, the state of a linear network can be
  asymptotically observed (resp. stabilized) via local
  computation. Consider for instance the problem of designing an
  observer \cite{GB-GM:91}, and let $C_j = e_j^\transpose$. Take $G =
  - A_j$, where $A_j$ denotes the $j$-th column of $A$. Notice that
  the matrix $A+GC_j$ can be written as a block-triangular matrix, and 
  it is stable because of Lemma \ref{quasi_stoc}. Finally, since the
  nonzero entries of $G$ correspond to the out-neighbors\footnote{The
    agent $i$ is an out-neighbor of $j$ if the $(i,j)$-th entry of $A$
    is nonzero, or, equivalently, if $(j,i)$ belongs to the edge set.}
  of the node $j$, the output injection operation $GC_j$ only requires
  local information.  \oprocend
\end{remark}

A class of undetectable attacks is now presented. Notice that
misbehaving agents can arbitrarily change their initial state without
being detected during the consensus iterations, and, by doing so,
misbehaving components can cause at most a constant error on the final
consensus value. Indeed, let $A$ be a consensus matrix, and let $K$ be
the set of misbehaving agents. Let $x(0)$ be the network initial
state, and suppose that the agents $K$ alter their initial value, so
that the network initial state becomes $x(0)+B_K c$, where $c \in
\mathbb{R}^{|K|}$. Recall from \cite{CDM:01} that $\lim_{t \rightarrow
  \infty} A^t = \textbf{1}\pi$, where \textbf{1} is the vector of all
ones, and $\pi$ is such that $\pi A = \pi$. Therefore , the effect of
the misbehaving set $K$ on the final consensus state is $\textbf{1}\pi
B_K c$. Clearly, if the vector $x(0)+B_K c$ is a valid initial state,
the misbehaving agents cannot be detected. On the other hand, since it
is possible for uncompromised nodes to estimate the observable part of
the initial state of the whole network, if an acceptability region (or
an a priori probability distribution) is available on initial states,
then, by analyzing the reconstructed state, a form of intrusion
detection can be applied, e.g., see~\cite{EM-SM-SS:10}. We conclude
this paragraph by showing that, if the misbehaving vector $B_K c$
belongs to the unobservability subspace of $(A,C_j)$, for some $j$,
then the misbehaving agents do not alter the final consensus
value. Let $v$ be an eigenvector associated with the unobservable
eigenvalue $\bar z$, i.e., $(\bar z I -A) v = 0$ and $C_j v = 0$. We
have $\pi (\bar z I -A) v = (\bar z - 1) \pi v = 0$, and, because of
the detectability of $(A,C_j)$, $| \bar z | < 1$ (cf. Lemma
\ref{detectability}). Hence $\pi v = 0$. Therefore, if the attack $B_K
c$ is unobservable from any agent, then $\lim_{t \rightarrow \infty}
A^t B_K c = \textbf{1}\pi B_K c = 0$, so that the change of the
initial states of misbehaving agents does not affect the final
consensus value.

A different class of unidentifiable attacks consists of injecting a
signal corresponding to an input-zero for the current network
state. We start by characterizing the potential disruption caused by
misbehaving nodes that introduce nonzero, but exponentially vanishing
inputs.\footnote{An output-zeroing input can always be written as $u
  (k) = -(CA^{\nu} B)^\dag C A^{\nu +1} (K_\nu A)^k x(0) - (CA^\nu
  B)^\dag C A^{\nu +1} \left(\sum_{l = 0}^{k-1} (K_\nu A)^{k-1-l} B
    u_h (l) \right) + u_h (h)$, where $\nu \in \natural$, $(CA^\nu B)$
  is the first nonzero Markov parameter, $K_\nu = I - B(C A^{\nu}
  B)^\dag C A^\nu$ is a projection matrix, $x(0) \in \bigcap_{l=0}^\nu
  \Ker(C A^l)$ is the system initial state, and $u_h (k)$ is such that
  $C A^\nu B u_h (k) = 0$ \cite{JT:06}.}

\begin{lemma}[Exponentially stable input]\label{summable_error}
  Let $A$ be a consensus matrix, and let $K$ be a set of agents. Let
  $u : \natural \mapsto \real^{|K|}$ be exponentially decaying. There
  exists $z \in (0,1)$ and $\bar u \in \real^{|K|}$ such that
  \begin{align*}
    \lim_{t \rightarrow \infty} \sum_{\tau =0}^t A^{t-\tau} B_K
    u(\tau) \preceq (1-z)^{-1} \textbf{1} \pi B_K \bar u,
  \end{align*}
  where $\preceq$ denotes component-wise inequality, $\textbf{1}$ is
  the vector of all ones of appropriate dimension, and $\pi$ is such
  that $\pi A = \pi$.
\end{lemma}
\begin{proof}
  Let $z \in (0,1)$ and $0 \preceq u_0 \in \real^{|K|}$ be such that
  $u(k) \preceq z^k u_0$. Then, since $A$ is a nonnegative matrix, for
  all $t,\tau \in \natural$, with $t \ge \tau$, we have $A^{t-\tau}
  B_K u(\tau) \preceq A^{t-\tau} B_K z^\tau u_0$, and hence $\lim_{t
    \rightarrow \infty} \sum_{\tau =0}^t A^{t-\tau} B_K u(\tau)
  \preceq \lim_{t \rightarrow \infty} \sum_{\tau =0}^t A^{t-\tau} B_K
  z^\tau u_0$. Notice that $(1-z)^{-1} = \lim_{t \rightarrow \infty
  }\sum_{\tau=0}^t z^\tau$. We now show that
    $\lim_{t \rightarrow \infty} \sum_{\tau =0}^t z^\tau (\textbf{1}
    \pi -A^{t-\tau}) = \lim_{t \rightarrow \infty} \sum_{\tau =0}^t
    E(t,\tau) \preceq 0$,
  from which the theorem follows. Let $e(t,\tau)$ be any component of
  $E(t,\tau)$. Because $\lim_{t \rightarrow \infty} A^t = \textbf{1}
  \pi$, there exist $c$ and $\rho$, with $|z| \le |\rho | < 1$, such
  that $e(t,\tau) \le c z^\tau \rho^{t-\tau}$. We have 
  \begin{align*}
    \lim_{t \rightarrow \infty} \sum_{\tau = 0}^t c z^\tau
    \rho^{t-\tau} = \lim_{t \rightarrow \infty} c\rho^t \sum_{\tau =
      0}^t z^\tau\rho^{-\tau} = 0,
  \end{align*}
  so that $\sum_{\tau = 0}^t E(t,\tau)$ converges to zero as $t$
  approaches infinity.
\end{proof}


Following Lemma \ref{summable_error}, if the zero dynamics are
exponentially stable, then misbehaving agents can affect the final
consensus value by a constant amount without being detected, if and
only if they inject vanishing inputs along input-zero directions. If
an admissible region is known for the network state, then a tight
bound on the effect of misbehaving agents injecting vanishing inputs
can be provided. Notice moreover that, in this situation, a
well-behaving agent is able to detect misbehaving agents whose state
is outside an admissible region by simply analyzing its
state. Finally, for certain consensus networks, the effect of an
exponentially stable input decreases to zero with the cardinality of
the network. Indeed, let $\pi = \bar \pi / n$, where $\bar \pi$ is a
constant row vector and $n$ denotes the cardinality of the
network. For instance, if $A$ is doubly stochastic, then $\pi =
\textbf{1}^\transpose / n$ \cite{CDM:01}. Then, when $n$ grows, the
effect of the input $u(t) = z^t \bar u$, with $|z| < 1$, on the
consensus value becomes negligible.

The left-invertibility and the stability of the zero dynamics are not
an inherent property of a consensus system. Consider for instance the
graph of Fig. \ref{unstable_zeros}, where the agents $\{1,2\}$ are
malicious. If the network matrices are
\begin{figure}[tb]
  \centering \subfigure[]{
    \includegraphics[width=.466\columnwidth]{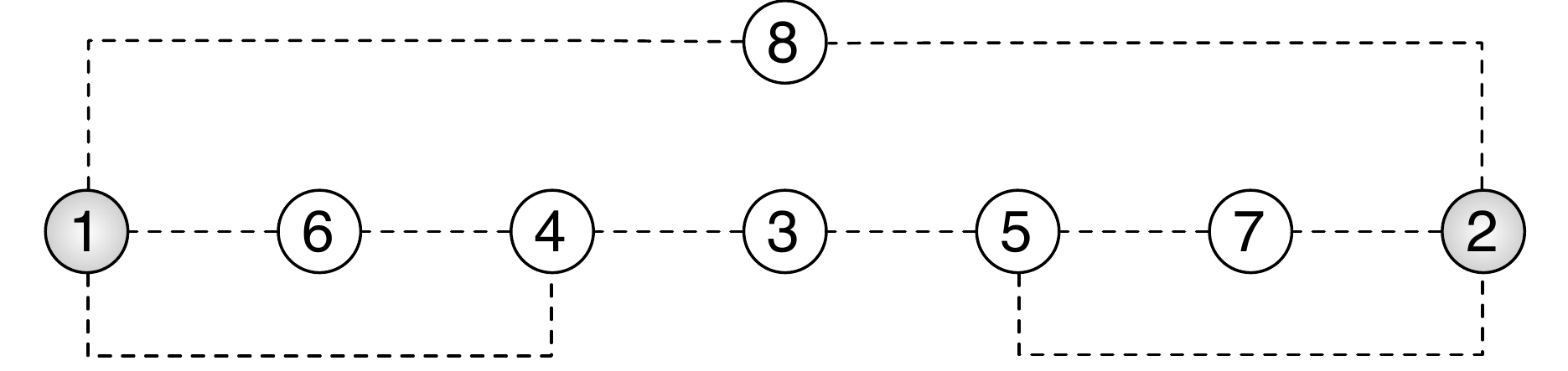}
    \label{unstable_zeros}
  } \subfigure[]{
    \includegraphics[width=.466\columnwidth]{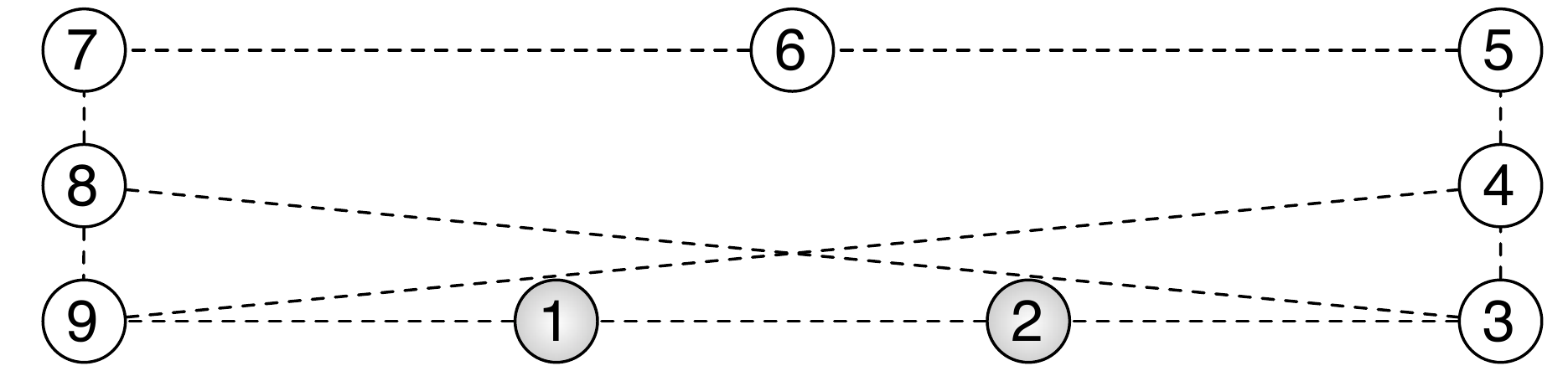}
    \label{unstable_circle}
  }
  \caption[Optional caption for list of figures]{In
    Fig. \ref{unstable_zeros} The agents $\{1,2\}$ are
    misbehaving. The consensus system $(A,B_{\{1,2\}},C_3)$ has
    unstable zeros. In Fig. \ref{unstable_circle} the agents $\{1,2\}$
    are misbehaving. The consensus system $(A,B_{\{1,2\}},C_6)$ is not
    left-invertible.}
\end{figure}
\begin{align*}
  A=\left[\begin{smallmatrix}
      1/2 & 0 & 1/2 & 0 & 0 & 0 & 0 & 0\\
      0 & 1/2 & 0 & 0 & 0 & 1/2 & 0 & 0\\
      0 & 0 & 1/3 & 1/3 & 1/3 & 0 & 0 & 0\\
      1/16 & 0 & 5/8 & 1/16 & 0 & 1/4 & 0 & 0\\
      0 & 1/16 & 1/4 & 0 & 5/16 & 0 & 3/8 & 0\\
      1/2 & 0 & 0 & 1/2 & 0 & 0 & 0 & 0\\
      0 & 1/3 & 0 & 0 & 2/3 & 0 & 0 & 0\\
      1/2 & 1/2 & 0 & 0 & 0 & 0 & 0 & 0
      \end{smallmatrix}
      \right],\, 
          B_{\{1,2\}}=\left[
            \begin{smallmatrix}
              1 & 0\\
              0 & 1\\
              0 & 0\\
              0 & 0\\
              0 & 0\\
              0 & 0\\
              0 & 0\\
              0 & 0
            \end{smallmatrix}
            \right], \,
      C_3=\left[
        \begin{smallmatrix}
         0 & 0 & 1 & 0 & 0 & 0 & 0 & 0\\
         0 & 0 & 0 & 1 & 0 & 0 & 0 & 0\\
         0 & 0 & 0 & 0 & 1 & 0 & 0 & 0
          \end{smallmatrix}
          \right],
\end{align*}
then the system $(A,B_{\{1,2\}},C_3)$ is left-invertible, but the
invariant zeros are $\{0,+2,-2\}$. Hence, for some initial
conditions, there exist non vanishing input sequences that do not
appear in the output. Moreover, for the graph in
Fig. \ref{unstable_circle}, let the network matrices be
\begin{align*}
  A&=\left[\begin{smallmatrix}
    1/3 &1/3 &0 &0 &0 &0 &0 &0 &1/3\\
    1/3 &1/3 &1/3 &0 &0 &0 &0 &0 &0\\
    0 &1/4 &1/4 &1/4 &0 &0 &0 &1/4 &0\\
    0 &0 &1/4 &1/4 &1/4 &0 &0 &0 &1/4\\
    0 &0 &0 &1/3 &1/3 &1/3 &0 &0 &0\\
    0 &0 &0 &0 &1/3 &1/3 &1/3 &0 &0\\
    0 &0 &0 &0 &0 &1/3 &1/3 &1/3 &0\\
    0 &0 &1/4 &0 &0 &0 &1/4 &1/4 &1/4\\
    1/4 &0 &0 &1/4 &0 &0 &0 &1/4 &1/4
      \end{smallmatrix}\right],\, 
    B_{\{1,2\}}=\left[\begin{smallmatrix}
    1 &0\\
    0 &1\\
    0 &0\\
    0 &0\\
    0 &0\\
    0 &0\\
    0 &0\\
    0 &0\\
    0 &0
    \end{smallmatrix}\right],\,
  C_6=\left[\begin{smallmatrix}
    0 &0 &0 &0 &1 &0 &0 &0 &0\\
    0 &0 &0 &0 &0 &1 &0 &0 &0\\
    0 &0 &0 &0 &0 &0 &1 &0 &0
     \end{smallmatrix}\right].
\end{align*}
It can be verified that the system $(A,B_{\{1,2\}},C_6)$ is not
left-invertible. Indeed, for zero initial conditions, any input of the
form $u_1=-u_2$ does not appear in the output sequence of the agent
$6$.
In some cases, the left-invertibility of a consensus system can be
asserted independently of the consensus matrix.

\begin{theorem}[Left-invertibility, single intruder
  case]\label{single_intruder} Let $A$ be a consensus matrix, and let
  $B_i = e_i$, $C_j = e_j^\transpose$. Then the system $(A,B_i,C_j)$ is
  left-invertible.
\end{theorem}
\begin{proof}
  Suppose, by contradiction, that $(A,B_i,C_j)$ is not
  left-invertible. Then there exist state trajectories that, starting
  from the origin, are invisible to the output. In other words, since
  the input is a scalar, the Markov parameters $C_jA^tB_i$ have to be
  zero for all $t$. Notice the $(i,k)$-th component of $A^t$ is
  nonzero if there exists a path of length $t$ from $i$ to
  $k$. Because $A$ is irreducible, there exists $t$ such that
  $C_jA^tB_i\neq 0$, and therefore the consensus system is
  left-invertible.
\end{proof}

If in Theorem~\ref{single_intruder} one identifies the $i$-th node
with a single intruder, and the $j$-th node with an observer node, the
theorem states that, for known initial conditions of the network, any
two distinct inputs generated by a single intruder produce different
outputs at all observing nodes, and hence can be detected. Consider
for example a flocking application, in which the agent are supposed to
agree on the velocity to be maintained during the execution of the
task \cite{WR-RWB-EMA:07}. Suppose that a linear consensus iteration
is used to compute a common velocity vector, and suppose that the
states of the agents are equal to each other. Then no single
misbehaving agent can change the velocity of the team without being
detected, because no zero dynamic can be generated by a single agent
starting from a consensus state.

We now consider the case in which several misbehaving agents are
allowed to act simultaneously. The following result relating the
position of the misbehaving agents in the network and the zero
dynamics of a consensus system can be asserted.

\begin{theorem}[Stability of zero dynamics]\label{stability_zero}
  Let $K$ be a set of agents and let $j$ be a network node. The zero
  dynamics of the consensus system $(A,B_K,C_j)$ are exponentially
  stable if one of the following is true:
\begin{enumerate}
\item \label{case1}the system $(A,B_K,C_j)$ is left-invertible, and there are no
  edges from the nodes $K$ to $V\setminus \{N_j\cup K\}$;
\item \label{case2}the system $(A,B_K,C_j)$ is left-invertible, and there are no
  edges from the nodes $V\setminus \{N_j\cup K\}$ to $N_j$; or
\item \label{case3}the sets $K$ and $N_j$ are such that $K \subseteq N_j$.
\end{enumerate}
\end{theorem}
\begin{proof}
  Let $z$ be an invariant zero, $x$ and $u$ a state-zero and
  input-zero direction, so that
  \begin{align}\label{cond1}
    (zI-A)x + B_K u = 0,\text{ and } C_j x = 0
  \end{align}
  Reorder the nodes such that the set $K$ comes first, the set $N_j
  \setminus K$ second, and the set $V\setminus \{K \cup N_j\}$
  third. The consensus matrix and the vector $x$ are accordingly
  partitioned as
  \begin{align*}
    A=
    \begin{bmatrix}
      A_{11} & A_{12} & A_{13}\\
      A_{21} & A_{22} & A_{23}\\
      A_{31} & A_{32} & A_{33}
    \end{bmatrix},\quad
    x=
    \begin{bmatrix}
        x_1\\
        x_2\\
        x_3
    \end{bmatrix},
  \end{align*}
  and the input and output matrices become $B_K=[I \; 0 \; 0]^\transpose$ and
  $C_j=[\ast \; I \; 0]$. For equations \eqref{cond1} to be verified, it has to be
  $x_2 = 0$, $z x_1 = A_{11} x_1 + A_{13} x_3 - u_k$, and
  \begin{align*}
     \begin{bmatrix}
       0\\
       z x_3
    \end{bmatrix}
    =
    \begin{bmatrix}
      A_{21} & A_{23}\\
      A_{31} & A_{33}
    \end{bmatrix}
    \begin{bmatrix}
      x_1\\
      x_3
    \end{bmatrix}
    .
  \end{align*}

  \textit{Case \ref{case1}.} Since there are no edges from the nodes
  $K$ to $V\setminus \{N_j\cup K\}$, we have $A_{31} = 0$, and hence
  it has to be $(zI - A_{33}) x_3 = 0$, i.e., $z$ needs to be an
  eigenvalue of $A_{33}$. We now show that $x_3 \neq 0$. Suppose by
  contradiction that $x_3 = 0$, and that $z$ is an invariant zero,
  with state-zero and input-zero direction $x = [x_1^\transpose \, 0
  \, 0]^\transpose$ and $u_K = (z I -A_{11}) x_1$, respectively. Then,
  for all complex value $\bar z$, the vectors $x$ and $u_K = (\bar z I
  - A_{11}) x_1$ constitute the state-zero and the input-zero
  direction associated with the invariant zero $\bar z$. Because the
  system is assumed to be left-invertible, there can only be a finite
  number of invariant zeros \cite{JT:06}, so that we conclude that
  $x_3 \neq 0$ or that the system has no zero dynamics. Because $z$
  needs to be an eigenvalue of $A_{33}$, and because of Lemma
  \ref{quasi_stoc}, we conclude that the zero dynamics are
  asymptotically stable.
  
  \textit{Case \ref{case2}.} Since there are no edges from the nodes
  $V\setminus \{N_j\cup K\}$ to $N_j$, we have $A_{23} = 0$. We now
  show that $\Ker(A_{21}) = 0$. Suppose by contradiction that $0 \neq
  x_1 \in \Ker(A_{21})$. Consider the equation $(zI - A_{33})x_3 =
  A_{31} x_1$, and notice that, because of Lemma \ref{quasi_stoc}, for
  all $z$ with $|z| \ge 1$, the matrix $zI - A_{33}$ is
  invertible. Therefore, if $|z| \ge 1$, the vector $[(x_1)^\transpose
  \, 0 \, ((zI - A_{33})^{-1} A_{31} x_1)^\transpose]^\transpose$ is a
  state-zero direction, with input-zero direction $u_K = -(zI-A_{11})
  x_1 + A_{13} x_3$. The system would have an infinite number of
  invariant zeros, being therefore not left-invertible. We conclude
  that $\Ker(A_{21}) = 0$. Consequently, we have $x_1 = 0$ and $(zI -
  A_{33}) x_3 = 0$, so that $|z| < 1$.

  \textit{Case \ref{case3}.} Reorder the variables such that the nodes
  $N_j$ come before $V\setminus N_j$. For the existence of a zero
  dynamics, it needs to hold $x_1 = 0$ and $(zI - A_{22}) x_2 = 0$.
  Hence, $|z| < 1$.
\end{proof}

\begin{figure}[tb]
  \subfigure[]{
    \includegraphics[width=.28\columnwidth]{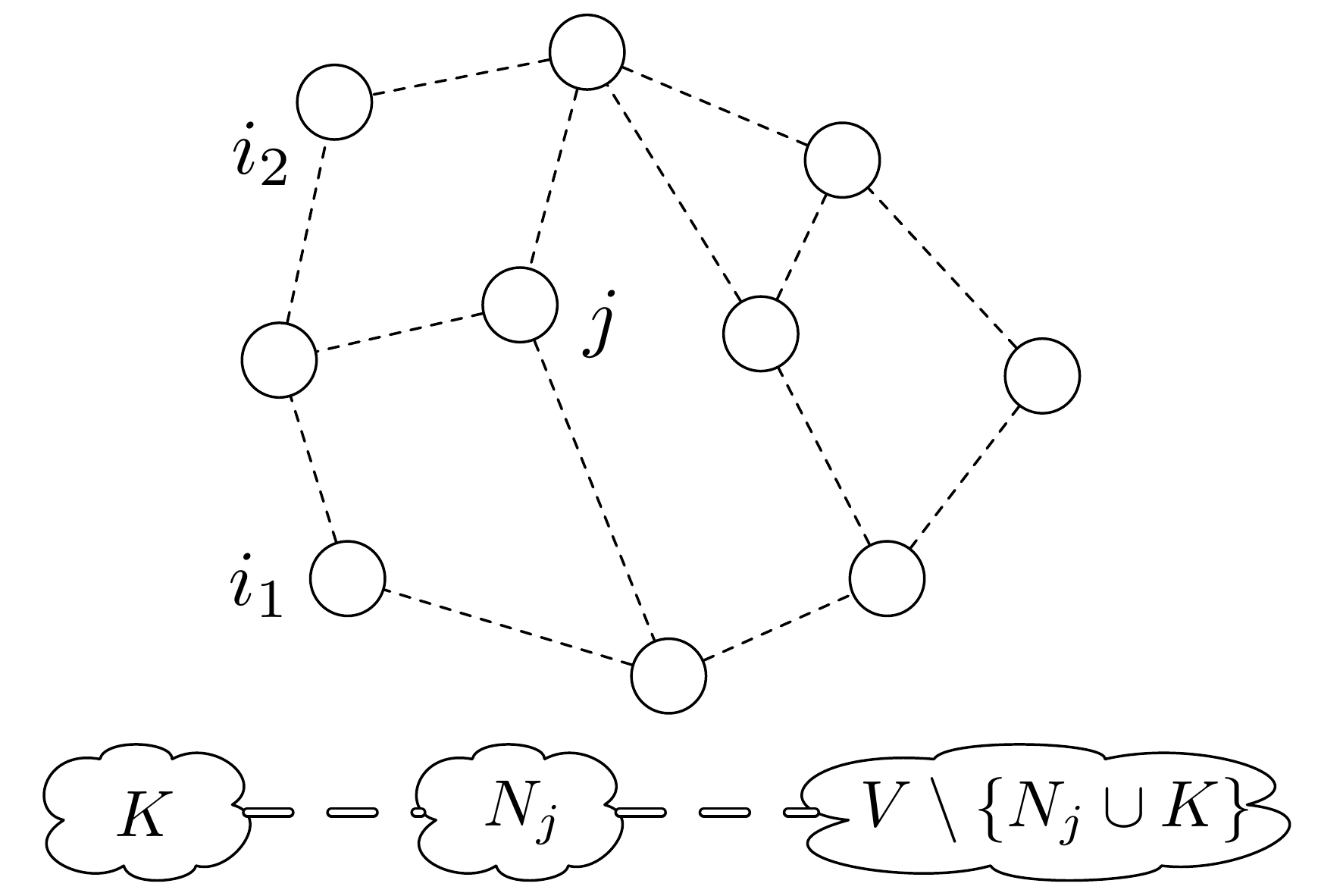}
    \label{cut_neigh}
  } \subfigure[]{
    \includegraphics[width=.28\columnwidth]{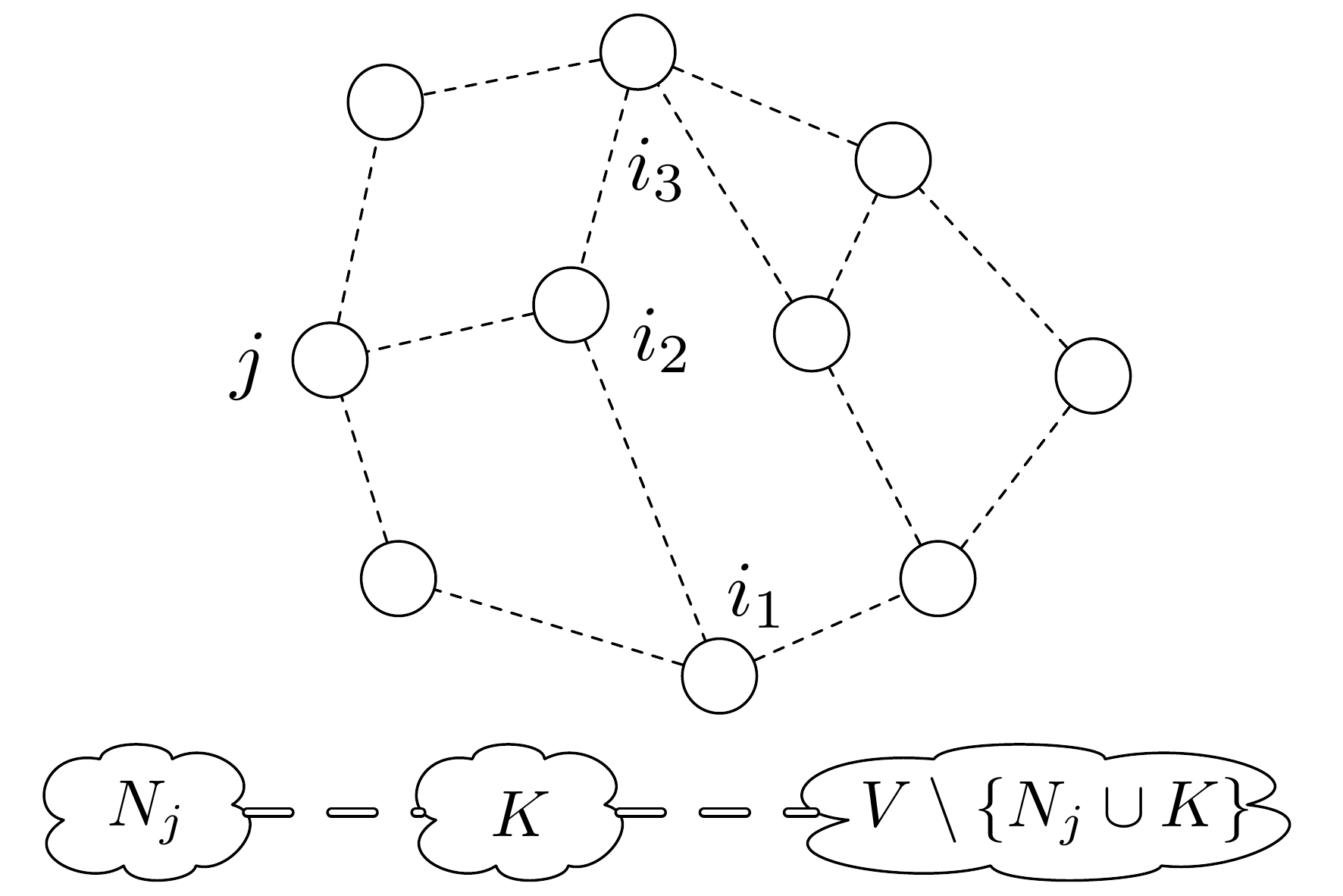}
    \label{int_cut}
  }
  \centering \subfigure[]{
    \includegraphics[width=.28\columnwidth]{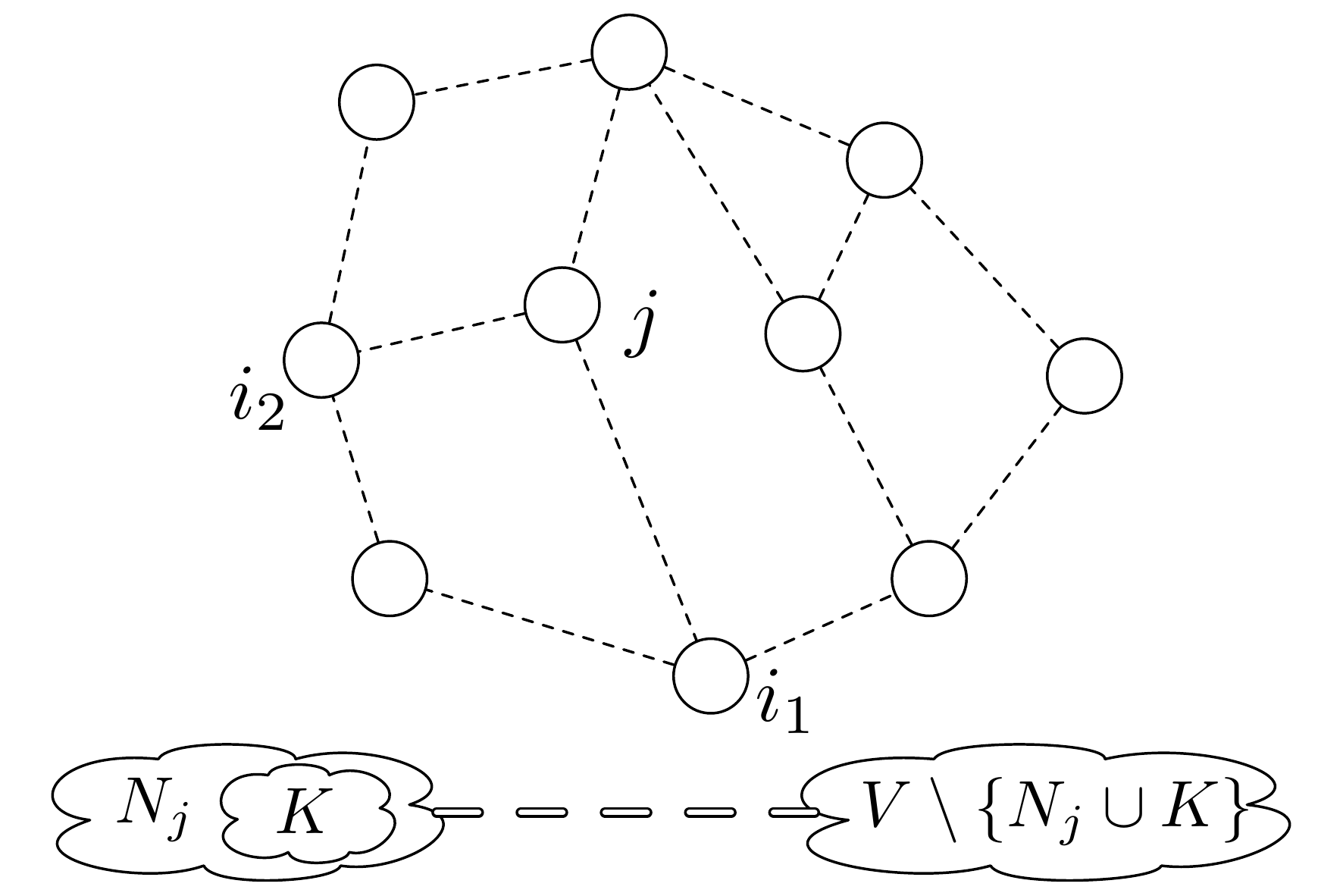}
    \label{intruders_neigh}
  }
  \caption[Optional caption for list of figures]{The stability of the
    zero dynamics of a left-invertible consensus system can be
    asserted depending upon the location of the misbehaving agents in
    the network. Let $j$ be the observer agent, and let $K$ be the
    misbehaving set. Then, the zero dynamics are asymptotically stable
    if the set $N_j$ separates the sets $K$ and $V\setminus \{N_j \cup
    K\}$ (cfr. Fig. \ref{cut_neigh}), or if the set $K$ separates the
    sets $N_j$ and $V\setminus \{N_j \cup K\}$
    (cfr. Fig. \ref{int_cut}), or if the set $K$ is a subset of $N_j$
    (cfr. Fig. \ref{intruders_neigh}).}
  \label{thm_stability}
\end{figure}

We are left to study the case of a network with zeros outside the open
unit disk, where intruders may inject non-vanishing inputs while
remaining unidentified. For this situation, we only remark that a
detection procedure based on an admissible region for the network
state can be implemented to detect inputs evolving along unstable zero
directions.

\section{Generic detection and identification of misbehaving
  agents}\label{generic_section}
In the framework of traditional control theory, the entries of the
matrices describing a dynamical system are assumed to be known without
uncertainties. It is often the case, however, that such entries only
approximate the exact values. In order to capture this modeling
uncertainty, \emph{structured systems} have been introduced and
studied, e.g., see \cite{JMD-CC-JW:03,WMW:85,KJR:88}. Let a structure
matrix $[M]$ be a matrix in which each entry is either a fixed zero or
an indeterminate parameter, and let the tuple of structure matrices
$([A],[B],[C],[D])$ denote the structured system
\begin{align}\label{structured}\begin{split}
  x(t+1)&=[A]x(t)+[B]u(t),\\
  y(t)&=[C]x(t)+[D]u(t).
\end{split}\end{align}
A numerical system $(A,B,C,D)$ is an admissible realization of
$([A],[B],[C],[D])$ if it can be obtained by fixing the indeterminate
entries of the structure matrices at some particular value, and two
systems are structurally equivalent if they are both an admissible
realization of the same structured system. Let $d$ be the number of
indeterminate entries altogether. By collecting the indeterminate
parameters into a vector, an admissible realization is mapped to a
point in the Euclidean space $\real^d$. A property which can be
asserted on a dynamical system is called \emph{structural} (or
\emph{generic}) if, informally, it holds for \emph{almost all}
admissible realizations. To be more precise, following \cite{KJR:88},
we say that a property is structural (or generic) if and only if the
set of admissible realizations satisfying such property forms a dense
subset of the parameters space.\footnote{A subset $S \subseteq P
  \subseteq \real^d$ is dense in $P$ if, for each $r \in P$ and every
  $\varepsilon > 0$, there exists $s \in S$ such that the Euclidean
  distance $\|s - r\| \le \varepsilon$.} Moreover, it can be shown
that, if a property holds generically, then the set of parameters for
which such property is not verified lies on an algebraic hypersurface
of $\real^d$, i.e., it has zero Lebesgue measure in the parameter
space. For instance, left-invertibility of a dynamical system is known
to be a structural property with respect to the parameters space
$\real^d$.

Let the connectivity of a structured system $([A],[B],[C])$ be the
connectivity of the graph defined by its nonzero parameters. In what
follows, we assume $[D]=0$, and we study the zero dynamics of a
structured consensus system as a function of its connectivity.
Let the generic rank of a structure matrix $[M]$ be the maximal rank
over all possible numerical realizations of $[M]$.

\begin{lemma}[Generic zero dynamics and
  connectivity]\label{generic_invariant_connectivity}
  Let $([A],[B],[C])$ be a $k$-connected structured system. If the
  generic rank of $[B]$ is less than $k$, then almost every numerical
  realization of $([A],[B],[C])$ has no zero dynamics.
\end{lemma}
\begin{proof}
  Since the system $([A],[B],[C])$ is $k$-connected and the generic
  rank $r$ of $[B]$ is less than $k$, there are $r$ disjoint paths
  from the input to the output \cite{JVDW:99}. Then, from Theorem 4.3
  in \cite{JVDW:99}, the system $([A],[B],[C])$ is generically
  left-invertible. Additionally, by using Lemma 3 in
  \cite{SS-CNH:08bis}, it can be shown that $([A],[B],[C])$ has
  generically no invariant zeros. We conclude that almost every
  realization of $([A],[B],[C])$ has no nontrivial zero dynamics.
\end{proof}

Given a structured triple $([A],[B],[C])$ with $d$ nonzero
elements, the set of parameters that make $([A],[B],[C])$ a consensus
system is a subset $S$ of $\mathbb{R}^{d}$, because the matrix
$A$ needs to be row stochastic and primitive. A certain property
that holds generically in $\mathbb{R}^{d}$ needs not be valid
generically with respect to the feasible set $S$.
Let $([A],[B],[C])$ be structure matrices, and let $S\subset
\mathbb{R}^{d}$ be the set of parameters that make
$([A],[B],[C])$ a consensus system. We next show that the
left-invertibility and the number of invariant zeros are generic
properties with respect to the parameter space $S$.

\begin{theorem}[Genericity of consensus
  systems]\label{genericity_consensus}
  Let $([A],[B],[C])$ be a $k$-connected structured system. If the
  generic rank of $[B]$ is less than $k$, then almost every consensus
  realization of $([A],[B],[C])$ has no zero dynamics.
\end{theorem}
\begin{proof}
  Let $d$ be the number of nonzero entries of the structured system
  $([A],[B],[C])$. From Theorem \ref{generic_invariant_connectivity}
  we know that, generically with respect to the parameter space
  $\mathbb{R}^d$, a numerical realization of $([A],[B],[C])$ has no
  zero dynamics. Let $S \subset \mathbb{R}^d$ be the subset of
  parameters that makes $([A],[B],[C])$ a consensus system. We want to
  show that the absence of zero dynamics is a generic property with
  respect to the parameter space $S$. Observe that $S$ is dense in
  $\mathbb{R}^{\bar d}$, where $\bar d \le d-n$ and $n$ is the
  dimension of $[A]$. Then \cite{ANK-SVF:75,KT:83}, it can be shown
  that, in order to prove that
  our property is generic with respect to $S$, it is sufficient to
  show that there exist some consensus systems which have no zero
  dynamics. To construct a consensus system with no zero dynamics
  consider the following procedure. Let $(A,B,C)$ be a nonnegative and
  irreducible linear system with no zero dynamics, where the number of
  inputs is strictly less that the connectivity of the associated
  graph. Notice that, following the above discussion, such system can
  always be found. The Perron-Frobenius Theorem for nonnegative
  matrices ensures the existence of a positive eigenvector $x$ of $A$
  associated with the eigenvalue of largest magnitude $r$
  \cite{CDM:01}. Let $D$ be the diagonal matrix whose main diagonal
  equals $x$, then the matrix $r^{-1}D^{-1}AD$ is a consensus matrix
  \cite{HM:88}. A change of coordinates of $(A,B,C)$ using $D$ yields
  the system $(D^{-1}AD,D^{-1}B,CD)$, which has no zero
  dynamics. Finally, the system $(r^{-1}D^{-1}AD,D^{-1}B,CD)$ is a
  $k$-connected consensus system with, generically, no zero
  dynamics. Indeed, if there exists a value $\bar z$, a state-zero
  direction $x_0$, and an input-zero direction $g$ for the system
  $(r^{-1}D^{-1}AD,D^{-1}B,CD)$, then the value $\bar z r$, with state
  direction $x_0/r$ and input direction $g$, is an invariant zero of
  $(D^{-1}AD,D^{-1}B,CD)$, which contradicts the hypothesis.
\end{proof}
Because a sufficiently connected consensus system has generically no
zero dynamics, the following remarks about the robustness of a generic
property should be considered. First, generic means open, i.e. some
appropriately small perturbations of the matrices of the system having
a generic property do not destroy this property. Second, generic
implies dense, hence any consensus system which does not have a
generic property can be changed into a system having this property
just by arbitrarily small perturbations. We are now able to state our
generic resilience results for consensus networks.

\begin{theorem}\emph{\bf (Generic identification of misbehaving agents)}
  Given a $k$-connected consensus network, generically, there exist
  only identifiable inputs for any set of $\lfloor \frac{k-1}{2}
  \rfloor$ misbehaving agents. Moreover, generically, there exist
  identifiable inputs for every set of $k-1$ misbehaving agents.
\end{theorem}
\begin{proof}
  Since $2\lfloor \frac{k-1}{2} \rfloor< k$, by Lemma
  \ref{generic_invariant_connectivity} the consensus system with any
  set of $2\lfloor \frac{k-1}{2} \rfloor$ has generically no zero
  dynamics. By Theorem \ref{identification}, any set of $\lfloor
  \frac{k-1}{2} \rfloor$ malicious agents is detectable and
  identifiable by every node in the network. We now consider the case
  of faulty agents. Let $V$ be the set of nodes, and $K_1,K_2\subset
  V$, with $|K_1|=|K_2|= k-1$, be two disjoint sets of faulty
  agents. Let $j\in V$. We need to show the existence of identifiable,
  i.e., faulty, inputs. By using a result of \cite{JVDW:99} on the
  generic rank of the matrix pencil of a structured system, since the
  given consensus network is $k$-connected and $|K_1| = k-1$, it can
  be shown that the system $(A,[B_{K_1} \enspace B_i],C_j)$, for all
  $i\in K_2$, is left-invertible, which confirms the existence of
  identifiable inputs for the current network state. By Definition
  \ref{malicious_definition}, we conclude that the faulty set $K_1$ is
  generically identifiable by any well-behaving agent.
\end{proof}

In other words, in a $k$-connected network, up to $\lfloor
\frac{k-1}{2} \rfloor$ (resp. $k-1$) malicious (resp. faulty) agents
are generically identifiable by every well behaving
agent. Analogously, it can be shown that generically up to $k-1$
misbehaving agents are generically detectable. In the next
section, we describe three algorithms to detect and identify
misbehaving agents.

\section{Intrusion detection algorithms}\label{approximate}
In this section we present three decentralized algorithms to detect
and identify misbehaving agents in a consensus network. Although the
first two algorithms require only local measurements, the complete
knowledge of the consensus network is necessary for the
implementation. The third algorithm, instead, requires the agents to
know only a certain neighborhood of the consensus graph, and it allows
for a local identification of misbehaving agents. As it will be clear
in the sequel, the third algorithm overcomes, under a reasonable set
of assumptions, the limitations inherent to centralized detection and
identification procedures.

Our first algorithm is based upon the following result.
\begin{theorem}[Detection filter]\label{detection_filter}
  Let $K$ be the set of misbehaving agents. Assume that the zero
  dynamics of the consensus system $(A,B_K,C_j)$ are exponentially
  stable, for some $j$. Let $A_{N_j}$ denote the $N_j$ columns of the
  matrix $A$. The filter
\begin{align}\begin{split}\label{filter}
z(t+1) &= (A+GC_{j})z(t) - GC_j x(t), \\ 
\tilde{x}(t) &=Lz(t)  + HC_j x(t),
\end{split}\end{align}
with $G=-A_{N_{j}}$, $H= C_{j}^\transpose$, and $L=I - HC_{j}$,
is such that, in the limit for $t\rightarrow \infty$, the vector
$\tilde x(t+1)-A\tilde x(t)$ is nonzero only if the input $u_K(t)$ is
nonzero. Moreover, if $K\subset N_j$, then the filter \eqref{filter}
asymptotically estimates the state of the network, independent of the
behavior of the misbehaving agents $K$.
\end{theorem}
\begin{proof}
  Let $G=-A_{N_{j}}$, and consider the estimation error $e(t+1)
  =z(t+1)-x(t+1) =(A+GC_{j})e(t)-B_K u_K(t)$.
  Notice that $Le (t)=Lz (t) + C_j^\transpose C_j x (t) -x (t)$, and
  hence $\tilde x (t) = x(t) + L e(t)$. Consequently, $\tilde x(t+1) -
  A \tilde x(t) = B_K u_K(t) +L e(t+1) - A L e(t)$. By using Lemma
  \ref{quasi_stoc}, it is a straightforward matter to show that
  $(A+GC_j)$ is Schur stable. If $u_K(t) = 0$, then $\tilde x(t+1) - A
  \tilde x(t)$ converges to zero. Suppose now that $K\subseteq
  N_j$. The reachable set of $e$, i.e., the minimum $(A+GC_{j})$
  invariant containing $\mathcal{B}_{K}$, coincides with
  $\mathcal{B}_K$. Indeed $(A+GC_j)\mathcal{B}_K=\emptyset$. Since
  $\mathcal{B}_K \subseteq \Ker(L)$ by construction, the vectors
  $Le(t)$ and $\tilde x(t) - x(t)$ converge to zero.
\end{proof}

By means of the filter described in the above theorem, a distributed
intrusion detection procedure can be designed, see
\cite{FP-AB-FB:06v}. Here, each well-behaving agent only implements
one detection filter, making the asymptotic detection task
computationally easy to be accomplished. We remark that, since the
filter converges exponentially, an exponentially decaying input of
appropriate size may remain undetected (see Lemma \ref{summable_error}
for a characterization of the effect of exponentially vanishing inputs
on the final consensus value). For a finite time detection of
misbehaving agents, and for the identification of misbehaving
components, a more sophisticated algorithm is presented in Algorithm
\ref{algorithm_1}.
\begin{theorem}[Complete identification]
  Let $A$ be a consensus matrix, let $K$ be the set of misbehaving
  agents, and let $c$ be the connectivity of the consensus
  network. Assume that:
\begin{enumerate}
  \item every agent knows the matrix $A$ and $k \ge |K|$, and
  \item $k < c$, if the set $K$ is faulty, and $2k < c$ if the set $K$
    is malicious.
\end{enumerate}
Then the Complete Identification algorithm allows each well-behaving
agent to generically detect and identify every misbehaving agent in
finite time.
\end{theorem}
\begin{proof}
  We focus on agent $j$. Let $k=|K|$, and let $\mathcal{K}$ be the set
  containing all the $\binom{n-1}{k+1}$ combinations of $k+1$ elements
  of $V\setminus \{j\}$. For each set $\tilde K\in \mathcal{K}$,
  consider the system $\Sigma_{\tilde{K}}=(A,B_{\tilde{K}},C_j)$, and
  compute\footnote{We refer the interested reader to
    \cite{MAM-GCV-ASW-CM:89} for a design procedure of a dead beat
    residual generator. Notice that the possibility of detecting and
    identifying the misbehaving agents is, as discussed in Section
    \ref{resilience_section} and \ref{generic_section}, guaranteed by
    the absence of zero dynamics in the consensus system.} a set of
  residual generator filters for $\Sigma_{\tilde{K}}$. If the
  connectivity of the communication graph is sufficiently high, then,
  generically, each residual function is nonzero if and only if the
  corresponding input is nonzero. Let $K$ be the set of misbehaving
  nodes, then, whenever $K\subset \tilde{K}$, the residual function
  associated with the input $\tilde{K}\setminus K$ becomes zero after
  an initial transient, so that the agent $\tilde{K}\setminus K$ is
  recognized as well-behaving. By exclusion, because the residuals
  associated with the misbehaving agents are always nonzero, the set
  $K$ is identified.
\end{proof}

By means of the Complete Identification algorithm, the detection and
the identification of the misbehaving agents take place in finite
time, because the residual generators can be designed as dead-beat
filters, and independent of the misbehaving input. It should be
noticed that, although no communication overhead is introduced in the
consensus protocol, the Complete Identification procedure relies on
strong assumptions. First, each agent needs to know the entire graph
topology, and second, the number of residual generators that each node
needs to design is proportional to $\binom{n-1}{k}$. Because an agent
needs to update these filters after each communication round, when the
cardinality of the network grows, the computational burden may
overcome the capabilities of the agents, making this procedure
inapplicable.

\IncMargin{1em}
\begin{algorithm}[t]
  \SetKwInOut{Input}{Input} 
  \SetKwInOut{Require}{Require}
  \SetKwInOut{Title}{Algorithm}
  \SetAlgoCaptionSeparator{}

  \Input{$A$; $k\ge |K|$;}

  \Require{The connectivity of $A$ to be $k+1$, if $K$ is faulty, and
    $2k+1$ otherwise;}

  \BlankLine

  Compute the residual generators for every set
  of $k+1$ misbehaving agents\; \While{the misbehaving agents are
    unidentified}{ Exchange data with the neighbors\; Update the
    state\; Evaluate the residual functions\; \If{every
      $i_{\text{th}}$ residual is nonzero}{ Agent $i$ is recognized as
      misbehaving.}  }
  
  \caption{\textit{Complete Identification}$\quad$($j$-th agent)}
  \label{algorithm_1}
\end{algorithm}
\DecMargin{1em}

In the remaining part of this section, we present a computationally
efficient procedure that only assumes partial knowledge of the
consensus network but yet allows for a local identification of the
misbehaving agents. Let $A$ be a consensus matrix, and observe that it
can be written as $A_d+\varepsilon \Delta$, where $\| \Delta
\|_{\infty}=2$, $0 \le \varepsilon \le 1$, and $A_d$ is block diagonal
with a consensus matrix on each of the $N$ diagonal blocks. For
instance, let $A=[a_{kj}]$, and let $V_1,\dots,V_N$ be the subsets of
agents associated with the blocks. Then the matrix $A_d=[\bar a_{kj}]$
can be defined as
\begin{enumerate}
  \item $\bar a_{kj}=a_{kj}$ if $k\neq j$, and $k,j \in V_i$, $i\in \{1,\dots,N\}$,
  \item $\bar a_{kk}=1-\sum_{j\in V_i, j\neq k} a_{kj}$, and
  \item $\bar a_{kj}=0$ otherwise.
\end{enumerate}
Moreover, $\Delta=2(A-A_d)/\| (A-A_d) \|_{\infty}$, and
$\varepsilon=\frac{1}{2}\| A-A_d \|_{\infty}$. Note that, if
$\varepsilon$ is ``small'', then the agents belonging to different
groups are weakly coupled. We assume the groups of weakly coupled
agents to be given, and we leave the problem of finding such
partitions as the subject of future research, for which the ideas
presented in \cite{RP-PK:81, JHC-JC-RAW:84} constitute a very relevant
result.

We now focus on the $h$-th block. Let $K=v\cup l$ be the set of
misbehaving agents, where $v=V_h \cap K$, and $l=K\setminus v$. Assume
that the set $v$ is identifiable by agent $j \in V_h$ (see Section
\ref{resilience_section}). Then, agent $j$ can identify the set $v$ by
means of a set of residual generators, each one designed to decouple a
different set of $|v|+1$ inputs.
 To be more precise, let $i \in V_h \setminus v$, and
consider the system
\begin{align}\label{res_v}\begin{split}
    \begin{bmatrix}
      x\\
      w_v
    \end{bmatrix}
    ^+&=
    \begin{bmatrix}
      A_d & 0\\
      E_vC_j & F_v
    \end{bmatrix}
    \begin{bmatrix}
      x\\
      w_v
    \end{bmatrix}
    +
    \begin{bmatrix}
      B_v & B_{i}\\
      0 & 0
    \end{bmatrix}
    \begin{bmatrix}
      u_v\\
      u_i
    \end{bmatrix}
    ,\\
    r_v&=
    \begin{bmatrix}
      H_vC_j & M_v
    \end{bmatrix}
    \begin{bmatrix}
      x\\
      w_v
    \end{bmatrix}
    ,
\end{split}\end{align}
and the system
\begin{align}\label{res_i}\begin{split}
    \begin{bmatrix}
      x\\
      w_i
    \end{bmatrix}
    ^+&=
    \begin{bmatrix}
      A_d & 0\\
      E_i C_j & F_i
    \end{bmatrix}
    \begin{bmatrix}
      x\\
      w_i
    \end{bmatrix}
    +
    \begin{bmatrix}
      B_v & B_{i}\\
      0 & 0
    \end{bmatrix}
    \begin{bmatrix}
      u_v\\
      u_i
    \end{bmatrix}
    ,\\
    r_i&=
    \begin{bmatrix}
      H_i C_j & M_i
    \end{bmatrix}
    \begin{bmatrix}
      x\\
      w_i
    \end{bmatrix}
    ,
\end{split}\end{align}
where the quadruple $(F_v,E_v,M_v,H_v)$ (resp. $(F_i,E_i,M_i,H_i)$)
describes a filter of the form \eqref{res_gen}, and it is designed as
in \cite{MAM-GCV-ASW-CM:89}. Then the misbehaving agents $v$ are
identifiable by agent $j$ because $v$ is the only set such that, for
every $i \in V_h \setminus v$, it holds $r_v\not\equiv 0$ and $r_i
\equiv 0$ whenever $u_v \not\equiv 0$. It should be noticed that,
since $A_d$ is block diagonal, the residual generators to identify the
set $v$ can be designed by only knowing the $h$-th block of $A_d$, and
hence only a finite region of the original consensus network. By
applying the residual generators to the consensus system
$A_d+\varepsilon \Delta$ with misbehaving agents $K$ we get
\begin{align*}
  \begin{bmatrix}
      \hat x\\
      \hat w_v
  \end{bmatrix}
  ^+&=\bar A_{\varepsilon,v}
  \begin{bmatrix}
    \hat x\\
    \hat w_v
  \end{bmatrix}
  +
  \begin{bmatrix}
      B_v & B_l & B_i\\
      0 & 0 & 0
  \end{bmatrix}
    \begin{bmatrix}
      u_v\\
      u_l\\
      u_i
  \end{bmatrix}
  ,\\
  \hat r_v&=
  \begin{bmatrix}
      H_vC_j & M_v
  \end{bmatrix}
    \begin{bmatrix}
      \hat x\\
      \hat w_v
  \end{bmatrix}
  ,
\end{align*}
and
\begin{align*}
    \begin{bmatrix}
      \hat x\\
      \hat w_i
  \end{bmatrix}
  ^+&=\bar A_{\varepsilon,i}
  \begin{bmatrix}
    \hat x\\
    \hat w_i
  \end{bmatrix}
  +
  \begin{bmatrix}
      B_v & B_l & B_i\\
      0 & 0 & 0
  \end{bmatrix}
    \begin{bmatrix}
      u_v\\
      u_l\\
      u_i
  \end{bmatrix}
  ,\\
  \hat r_i&=
  \begin{bmatrix}
      H_iC_j & M_i
  \end{bmatrix}
  \begin{bmatrix}
    \hat x\\
    \hat w_i
  \end{bmatrix}
  ,
\end{align*}
where
\begin{align*}
  \bar A_{\varepsilon,v}= 
  \begin{bmatrix}
      A_d+\varepsilon \Delta & 0\\
      E_vC_j & F_v
  \end{bmatrix}
  , \enspace
  \bar A_{\varepsilon,i}=
  \begin{bmatrix}
    A_d+\varepsilon \Delta & 0\\
    E_iC_j & F_i
  \end{bmatrix}
  .
\end{align*}
Because of the matrix $\Delta$ and the input $u_l (t)$, the residual
$r_i (t)$ is generally nonzero even if $u_i \equiv 0$. However, the
misbehaving agents $v$ remain identifiable by $j$ if for each $i\in
V_h \setminus v$ we have $\|\hat r_v\|_{\infty} > \|\hat
r_i\|_{\infty}$ for all $u_v \not\equiv 0$.

\IncMargin{1em}
\begin{algorithm}[t]
  \SetKwInOut{Input}{Input}
  \SetKwInOut{Require}{Require}
  \SetKwInOut{Set}{Define}
  \SetKwInOut{Title}{Algorithm}


  \Input{$A_h$; $k_j\ge |K \cap V_h|$; threshold $T_h$}

  \Require{The connectivity of $A_d^j$ to be $k_j+1$, if $K$ is
    faulty, and $2 k_j+1$ otherwise;}

  \BlankLine

  \While{the misbehaving agents are unidentified}{

    Exchange data with the neighbors\; 

    Update the state\;

    Evaluate the residual functions\;

    \If{$i_{\text{th}}$ residual is greater than $T_h$}{ Agent
      $i$ is recognized as misbehaving.}

    }
  \SetAlgoCaptionSeparator{}
  \label{algo:local}
  \caption{\textit{Local Identification}$\quad$($j$-th agent)}
\end{algorithm}
\DecMargin{1em}

\begin{theorem}[Local identification]\label{local_identification}
  Let $V$ be the set of agents, let $K$ be the set of misbehaving
  agents, and let $A_d+\varepsilon \Delta$ be a consensus matrix,
  where $A_d$ is block diagonal, $\|\Delta\|_{\infty}=2$, and $0 \le
  \varepsilon \le 1$. Let each block $h$ of $A_d$ be a consensus
  matrix with agents $V_h\subseteq V$, and with connectivity $|K \cap
  V_h|+1$. There exists $\alpha >0$ and $\subscr{u}{max}\ge 0$, such
  that, if each input signal $u_i (t)$, $i \in K$, takes value in
  $\mathcal{U}=\{u : \varepsilon \alpha \subscr{u}{max} \le
  \|u\|_{\infty} \le \subscr{u}{max}\}$,\footnote{The norm
    $\|u\|_\infty$ is intended in the vector sense at every instant of
    time. The misbehaving input is here assumed to be nonzero at every
    instant of time.} then each well-behaving agent $j\in V_h$
  identifies in finite time the faulty agents $K\cap V_h$ by means of
  the Local Identification algorithm.
\end{theorem}
\begin{proof}
    We focus on the agent $j \in V_h$, and, without loss of generality,
  we assume that $u_K(0)\neq 0$, and that the residual generators have
  a finite impulse response. Let $d_j=\|V_h\|$, and note that $d_j$
  time steps are sufficient for each agent $j \in V_h$ to identify the
  misbehaving agents. Let $u^t$ denote the input sequence up to time
  $t$. Let $v=K\cap V_h$, $l=K\setminus v$, and observe that
$
  \hat r_v(d_j)=\left[
  \begin{smallmatrix}
    H_vC_j & M_v
  \end{smallmatrix}
  \right]
\bar A_{\varepsilon,v}^{d_j}\bar
  x(0) + \hat h_v \star u_v^{d_j-1} + \hat h_l \star u_l^{d_j-1}$,
where $\hat h_v$ and $\hat h_l$ denote the impulse response from $u_v$
and $u_l$ respectively, and $\star$ denotes the convolution
operator. We now determine an upper bound for each term of $\hat
r_v(d_j)$. Let the misbehaving inputs take value in $\mathcal{U}=\{u :
\varepsilon \alpha \subscr{u}{max} \le \|u\|_{\infty} \le
\subscr{u}{max}\}$. By using the triangle inequality on the impulse
responses of the residual generator, it can be shown that
  $\| \hat h_l \star u_l^{d_j-1} \|_{\infty} \le \| h_l \star u_l^{d_j-1}\|_{\infty} +\varepsilon c_1 \subscr{u}{max}= \varepsilon c_1 \subscr{u}{max}$,
where $h_l$ denotes the impulse response form $u_l$ to $r_v$ of the
system \eqref{res_v}, and $c_1$ is a finite positive constant
independent of $\varepsilon$. Moreover, it can be shown that there
exist two positive constant $c_2$ and $c_3$ such that
  $\|\left[
  \begin{smallmatrix}
    H_vC_j & M_v
  \end{smallmatrix}
  \right]
\bar A_{\varepsilon,v}^{d_j}\bar x(0)\|_{\infty}
 \le \varepsilon c_2 \subscr{u}{max},
$
and
  $\min_{u_v \in \mathcal{U}} \| \hat h_v \star u_v^{d_j-1}
  \|_{\infty} \ge \min_{u_v \in \mathcal{U}} \| h_v \star
  u_v^{d_j-1}\|_{\infty} -\varepsilon c_3 \subscr{u}{max}$.
Analogously, for the residual generator associated with the
well-behaving agent $i$, we have
$
  \hat r_i(d_j)=\left[
  \begin{smallmatrix}
    H_iC_j & M_i
  \end{smallmatrix}
  \right]
\bar A_{\varepsilon,i}^{d_j}\bar
  x(0) + \hat h_v \star u_v^{d_j-1} + \hat h_l \star u_l^{d_j-1}$,
and hence
$
  \hat r_i(d_j)\le \varepsilon (c_4^{(i)}+c_5^{(i)}+c_6^{(i)}) \subscr{u}{max}$.
Let $\bar c = c_ 1 + c_2 +c_3 + \max_{i\in V_h\setminus v}
(c_4^{(i)}+c_5^{(i)}+c_6^{(i)})$, and let $\beta$ be such that
$\min_{u_v \in \mathcal{U}} \| h_v \star u_v^{d_j-1}\|_{\infty} >
\beta \subscr{u}{min}$. Then a correct identification of the
misbehaving agents $v$ takes place if $ \beta \subscr{u}{min} = \beta
\varepsilon \alpha \subscr{u}{max}> \varepsilon \bar c
\subscr{u}{max}$, and hence if $\alpha > \bar c / \beta$.
\end{proof}

Notice that the constant $\alpha$ in Theorem
\ref{local_identification} can be computed by bounding the infinity
norm of the impulse response of the residual generators. An example is
in Section \ref{example_approx}. A procedure to achieve local
detection and identification of misbehaving agents is in Algorithm
\ref{algo:local}, where $A_d^h$ denotes the $h$-th block of $A_d$, and
$T_h$ the corresponding threshold value. Observe that in the Local
Identification procedure an agent only performs local computation, and
it is assumed to have only local knowledge of the network structure.

\begin{remark}
  It is a nontrivial fact that the misbehaving agents become locally
  identifiable depending on the magnitude of $\varepsilon$. Indeed, as
  long as $\varepsilon > 0$, the effect of the perturbation
  $\varepsilon \Delta$ on the residuals becomes eventually relevant
  and prevents, after a certain time, a correct identification of the
  misbehaviors \cite{RP-PK:81}.  \oprocend
\end{remark}

\section{Numerical examples}\label{examples}

\subsection{Complete detection and identification}
Consider the network of Fig. \ref{2rombi}, and let $A$ be a randomly
chosen consensus matrix. In particular, let
\begin{align*}
  A=\left[\begin{smallmatrix}
      0.2795  &  0.1628   &      0  &  0.1512 &   0.4066 &        0  &       0 & 0\\
      0.0143  &  0.3363   & 0.3469  &       0 &        0 &   0.3025  &       0 & 0\\
      0  &  0.0718   & 0.1904  &  0.2438 &        0 &        0  &  0.4941 & 0\\
      0.0844  &       0   & 0.4457  &  0.0660 &        0 &        0  &       0 & 0.4040\\
      0.1709  &       0   &      0  &       0 &   0.2694 &   0.2472  &       0 & 0.3125\\
      0  &  0.4199   &      0  &       0 &   0.1575 &   0.3293  &  0.0932 & 0\\
      0  &       0   & 0.0174  &       0 &        0 &   0.4241  &  0.2850 & 0.2735\\
      0 & 0 & 0 & 0.3024 & 0.2039 & 0 & 0.2065 & 0.2873
      \end{smallmatrix}\right].
\end{align*}
The network is $3$-connected, and it can be verified that for any set
$K$ of $3$ misbehaving agents, and for any observer node $j$, the
triple $(A,B_K,C_j)$ is left-invertible. Also, for any set $K$ of
cardinality $2$, and for any node $j$, the triple $(A,B_K,C_j)$ has no
invariant zeros. As previously discussed, any well-behaving node can
detect and identify up to $2$ faulty agents, or up to $1$ malicious
agent. Consider the observations of the agent $1$, and suppose that
the agents $\{3,7\}$ inject a random signal into the network. As
described in Algorithm \ref{algorithm_1}, the agent $1$ designs the
residual generator filters and computes the residual functions for
each of the $\binom{7}{3}$ possible sets of misbehaving nodes, and
identify the well-behaving agents. Consider for example the system
$x(t+1)=Ax(t)+B_3u_3(t)+B_4u_4(t)+B_7u_7(t)$, and suppose we want to
design a filter of the form \eqref{res_gen} which is only sensible to
the signal $u_4$. The unobservability subspace
$\mathcal{S}^M_{\{3,7\}}=(\mathcal{V}^{*}_{\{3,7\}} +
\mathcal{S}^{*}_{\{3,7\}})$, is
\begin{align*}
  \mathcal{S}^M_{\{3,7\}}=\Image\left(\left[\begin{smallmatrix}
      0  &  0 &   0 &    0 &         0\\
      0  &  0 & 0 &  -0.6624  &       0\\
      0  & 1  & 0  & 0  &       0\\
      0 &  0 &  -0.4740 &  -0.6597 &    0\\
      0   & 0  & -0.8798   & 0.3548   & 0\\
      0.4116 & 0 &  -0.0327 &   0.0132  &       0\\
      0        & 0        & 0        & 0  & 1 \\
      0.9114 & 0 & 0.0148 & -0.0060 & 0
        \end{smallmatrix}\right]\right),
\end{align*}
and a possible choice for the matrices of the residual generator is
\begin{align*}
  F&=\left[
    \begin{smallmatrix}
      0 & 0 & 0\\
      0.0014 & -0.3222 & -0.3424\\
      -0.0013 & 0.3031 & 0.3222
    \end{smallmatrix}
  \right],
  E=\left[
    \begin{smallmatrix}
      0.2795 &  0.1628&   0.1512&   0.4066\\
      0.0138  & 0.4982 &   -0.2280  & 0.2003\\
      0.0082&    -0.6095&   0.3012 &   -0.1568
    \end{smallmatrix}
  \right],\\
   M&=\left[
    \begin{smallmatrix}
      -1   &0   & 0\\
      0   & 0.9999   & 0.0128
    \end{smallmatrix}
  \right], \text{ and }
  H=\left[
   \begin{smallmatrix}
     1  &0  & 0  & 0\\
     0  &  -0.7491 &  0.5832 &   -0.3142
    \end{smallmatrix}
  \right].
\end{align*}
It can be checked that, independent of the initial condition of the
network, the residual function associated with the input $4$ is zero,
as in \ref{rombi}, so that the agent $4$ is regarded as
well-behaving. Agents $3$, $7$, instead, have always nonzero
residual functions, and are recognized as misbehaving.
\begin{figure}[tb]
  \centering \subfigure[]{
    \includegraphics[width=.15\columnwidth]{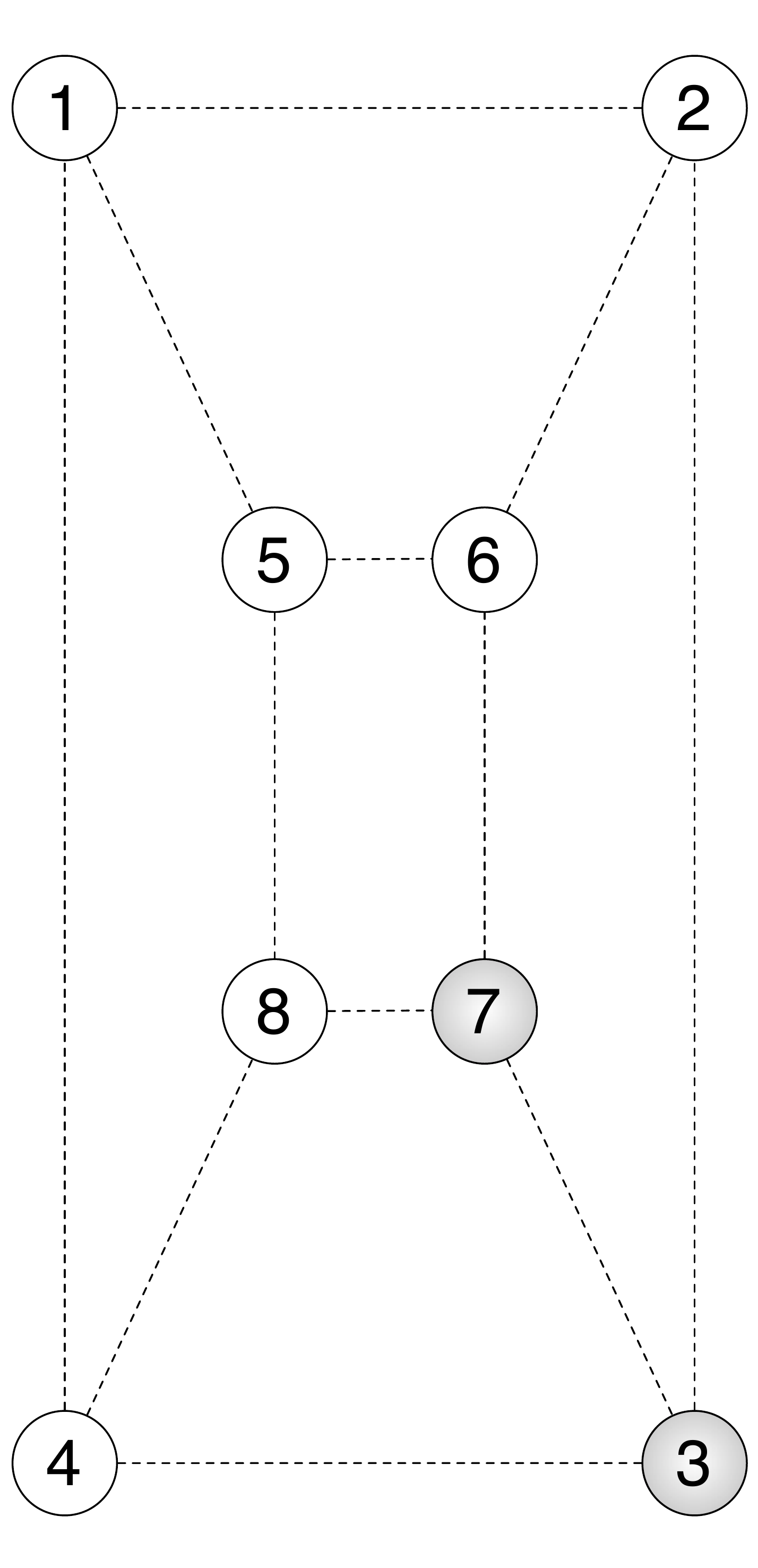}
    \label{2rombi}
  } \subfigure[]{
    \includegraphics[width=.36\columnwidth]{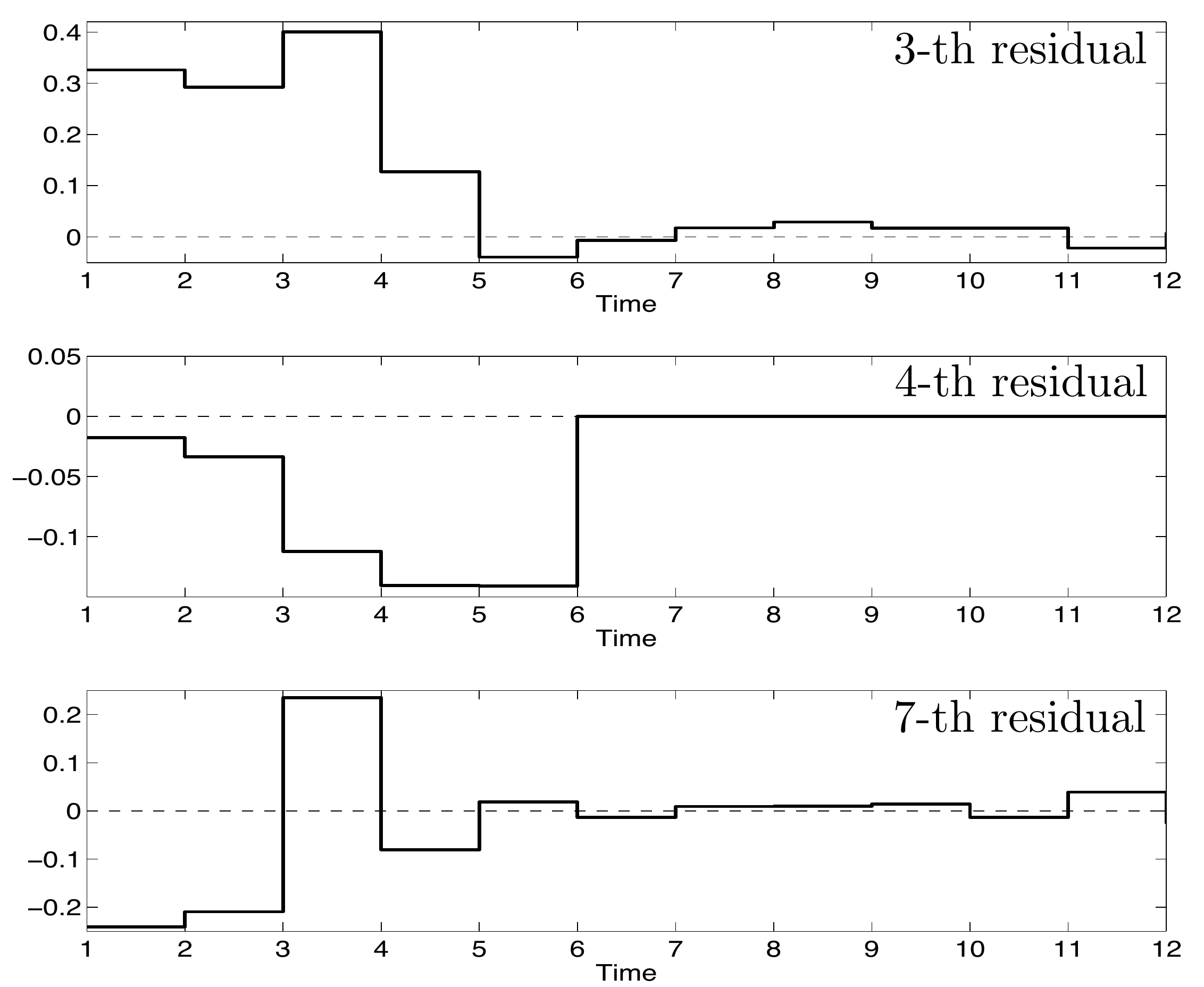}
    \label{rombi}
  }
 \caption[Optional caption for list of figures]{In Fig. \ref{2rombi} a
    consensus network where the nodes $3$ and $7$ are faulty. In
    Fig. \ref{rombi} the residual functions computed by the agent
    $1$ under the hypothesis that the misbehaving set is
    $\{3,4,7\}$.}
\end{figure}
If the misbehaving nodes are allowed to be malicious, then no more
than $1$ misbehaving node can be tolerated. Indeed, because of Theorem
\ref{left_connectivity}, there exists a set $\bar K$ of $4$
misbehaving agents such that the system $(A,B_{\bar K},C_1)$ exhibits
nontrivial zero dynamics. For instance, let $\bar K=\{2,4,6,8\}$, and
note that if the initial condition $x(0)$ belongs to
\begin{align*}
  \Vtar=\Image \left( \left[
    \begin{smallmatrix}
      0 & 0 & 0\\
      0 & 0 & 0\\
      1 & 0 & 0\\
      0 & 0 & 0\\
      0 & 0 & 0\\
      0 & 0.7842 & 0\\
      0 & 0 & 1\\
      0 & -0.6205 & 0
    \end{smallmatrix}
    \right]
  \right),
\end{align*}
then the input $u_K(t)=F_b x(t)$,\footnote{The malicious agents need
  to know the entire state to implement this feedback law. The case in
  which only local feedback is allowed is left as a direction for
  future research, for which the result in \cite{SS-CNH:08fault} is
  meaningful.} where
\begin{align*}
  F_b=\left[
    \begin{smallmatrix}
      0 & 0 & -0.3469 & 0 & 0 & -0.1860 & 0 & 0.1472\\
      0 & 0 & -0.4457 & 0 & 0 & 0.1966 & 0 & -0.1555\\
      0 & 0 & 0 & 0 & 0 & -0.1063 & -0.1148  & 0.0841\\
      0 & 0 & 0 & 0 & 0 & 0.0636 & -0.1894 & -0.0503
    \end{smallmatrix}
    \right],
\end{align*}
is such that $y_1(t)=0$ for all $t\ge 0$. Therefore, the two systems
$(A,B_{\{2,4\}},C_1)$ and $(A,B_{\{6,8\}},C_1)$, with initial
conditions $x_1(0)$ and $x_2(0)=x_1(0)-x(0)$, and inputs
\begin{align*}
  u_{\{2,4\}}(t)&=\left[
    \begin{smallmatrix}
      1 & 0 & 0 & 0\\
      0 & 1 & 0 & 0\\
    \end{smallmatrix}
    \right]F_b(x_1(t)-x_2(t)),\;
  u_{\{6,8\}}(t)=\left[
    \begin{smallmatrix}
      0 & 0 & 1 & 0\\
      0 & 0 & 0 & 1\\
    \end{smallmatrix}
    \right]F_b(x_2(t)-x_1(t)),
\end{align*}
have exactly the same output dynamics, so that the two sets $\{2,4\}$
and $\{6,8\}$ are indistinguishable by the agent $1$.

\subsection{Local detection and identification}\label{example_approx}
\begin{figure}
  \centering \subfigure[]{
    \includegraphics[width=.4\columnwidth]{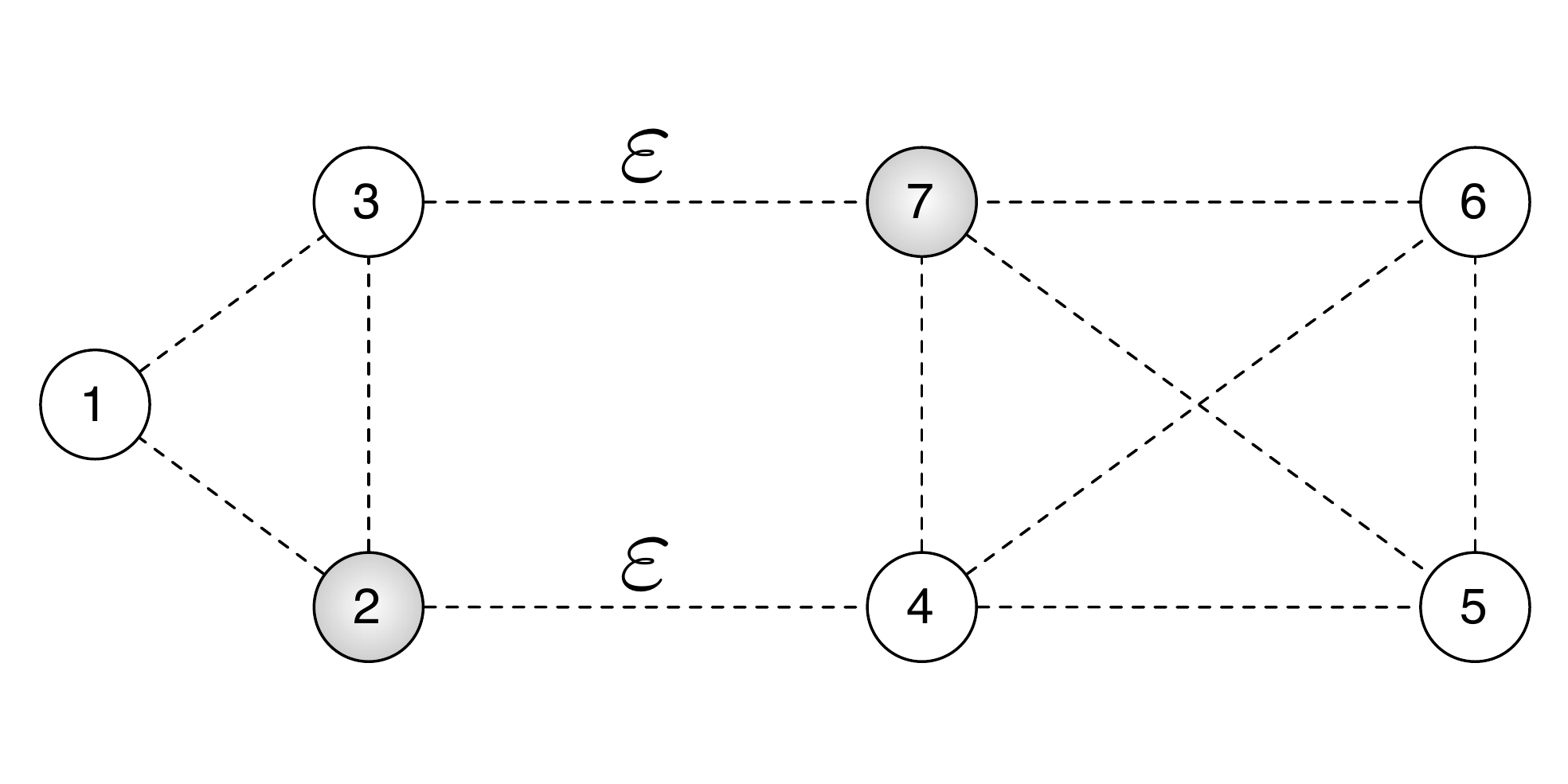}
    \label{approx_network}
  } \subfigure[]{
    \includegraphics[width=.4\columnwidth]{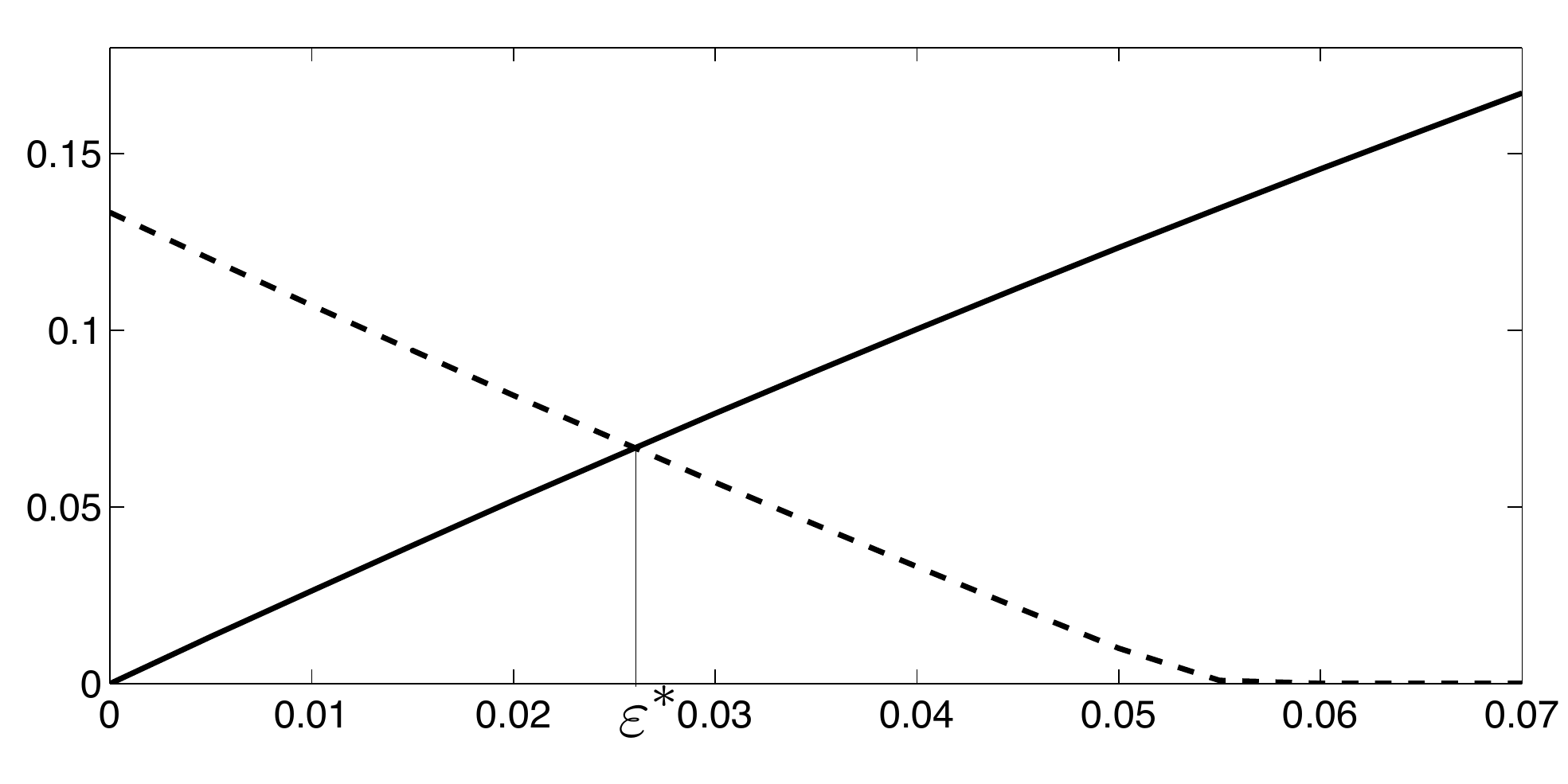}
    \label{maxmin}
  }
 \caption[Optional caption for list of figures]{In
    Fig. \ref{approx_network} a consensus network with weak
    connections. In Fig. \ref{maxmin} the solid line corresponds to
    the largest magnitude of the residual associated with the
    well-behaving agent $3$, while the dashed line denotes the
    smallest magnitude of the residual associated with the misbehaving
    agent $2$, both as a function of the parameter $\varepsilon$. If
    $\varepsilon\le \varepsilon^*$, then there exists a threshold that
    allows to identify the misbehaving agent $2$.}
\end{figure}
Consider the consensus network in Fig. \ref{approx_network}, where
$A=A_d+\varepsilon \Delta$, $\varepsilon \in \real$, $0 \le
\varepsilon \le 1$, and
\begin{align*}
  A_d=\left[
    \begin{smallmatrix}
      \frac{1}{3} & \frac{1}{3} & \frac{1}{3} & 0 & 0 & 0 & 0\\
      \frac{1}{3} & \frac{1}{3} & \frac{1}{3} & 0 & 0 & 0 & 0\\
      \frac{1}{3} & \frac{1}{3} & \frac{1}{3} & 0 & 0 & 0 & 0\\
      0 & 0 & 0 & \frac{1}{4} & \frac{1}{4} & \frac{1}{4} & \frac{1}{4}\\
      0 & 0 & 0 & \frac{1}{4} & \frac{1}{4} & \frac{1}{4} & \frac{1}{4}\\
      0 & 0 & 0 & \frac{1}{4} & \frac{1}{4} & \frac{1}{4} & \frac{1}{4}\\
      0 & 0 & 0 & \frac{1}{4} & \frac{1}{4} & \frac{1}{4} & \frac{1}{4}
    \end{smallmatrix}
    \right],
    \Delta=\left[
    \begin{smallmatrix}
      0 & 0 & 0 & 0 & 0 & 0 & 0\\
      0 & -1 & 0 & 1 & 0 & 0 & 0\\
      0 & 0 & -1 & 0 & 0 & 0 & 1\\
      0 & 0 & 1 & 0 & -1 & 0 & 0\\
      0 & 0 & 0 & 0 & 0 & 0 & 0\\
      0 & 0 & 0 & 0 & 0 & 0 & 0\\
      0 & 0 & 1 & 0 & 0 & 0 & -1
    \end{smallmatrix}
    \right].
\end{align*}
Let $K=\{2,7\}$ be the set of misbehaving agents, let $0.1 \le
u_2(t),u_7(t) \le 3$ at each time $t$, and let $\|x(0)\|_{\infty} \le
1$. Consider the agent $1$, and let $(F_2,E_2,M_2,H_2)$ and
$(F_3,E_3,M_3,H_3)$ be the residual generators as in \eqref{res_v} and
\eqref{res_i}, respectively, where
\begin{align*}
  F_2&=\left[
    \begin{smallmatrix}
      -1/3 & -1/3\\
      1/3 & 1/3
    \end{smallmatrix}
    \right],\enspace
    E_2=\left[
    \begin{smallmatrix}
      -2/3 & 0 & -1/3\\
      2/3 & 0 & 1/3
    \end{smallmatrix}
    \right],
    M_2=\left[
    \begin{smallmatrix}
      1 & 0\\
      0 & -1
    \end{smallmatrix}
    \right],\enspace
    H_2=\left[
    \begin{smallmatrix}
      1 & 0 & 0\\
      0 & 1 & 0
    \end{smallmatrix}
    \right],
\end{align*}
and
\begin{align*}
  F_3&=\left[
    \begin{smallmatrix}
      -1/3 & 1/3\\
      -1/3 & 1/3
    \end{smallmatrix}
    \right],\enspace
    E_3=\left[
    \begin{smallmatrix}
      -2/3 & -1/3 & 0\\
      -2/3 & -1/3 & 0
    \end{smallmatrix}
    \right],
    M_3=\left[
    \begin{smallmatrix}
      -1 & 0\\
      0 & 1
    \end{smallmatrix}
    \right],\enspace
    H_3=\left[
    \begin{smallmatrix}
      -1 & 0 & 0\\
      0 & 0 & 1
    \end{smallmatrix}
    \right].
\end{align*}
Let $\hat h_2^3$ (resp. $\hat h_7^3$) be the impulse response from the
input $u_2$ (resp. $u_7$) to $\hat r_3$, and let $u_2^1$
(resp. $u_7^1$) denote the input signal $u_2$ (resp. $u_7$) up to time
$1$. Note that the misbehaving agent can be identified after $2$ time
steps, and that the residual associated with the agent $3$ is
\begin{align*}
  \hat r_3(2)=
  \left[
  \begin{smallmatrix}
    H_3C_1 & M_3
  \end{smallmatrix}
  \right]
  \left[
  \begin{smallmatrix}
    A_d+\varepsilon\Delta & 0\\
    E_3C_1 & F_3
  \end{smallmatrix}
  \right]^2
  \left[
  \begin{smallmatrix}
    x(0)\\
    0
  \end{smallmatrix}
  \right]
  +\hat h_2^3 \star u_2^1 + \hat h_7^3 \star u_7^1,
\end{align*}
where $\star$ denotes the convolution operator. After some computation
we obtain
\begin{align*}
  \hat r_3(2)=\varepsilon \left[
  \begin{smallmatrix}
    H_3C_1 & M_3
  \end{smallmatrix}
  \right]
  \left[
  \begin{smallmatrix}
    A_d\Delta +\Delta A_d +\varepsilon  \Delta^2& \Delta B_2  & \Delta B_7\\
    E_3C_1\Delta & 0 & 0
  \end{smallmatrix}
  \right]
  \left[
  \begin{smallmatrix}
    x(0)\\
    u_2(0)\\
    u_7(0)
  \end{smallmatrix}
  \right]
\end{align*}
and, analogously,
\begin{align*}
  \hat r_2(2)&=\varepsilon \left[
  \begin{smallmatrix}
    H_2C_1 & M_2
  \end{smallmatrix}
  \right]
  \left[
  \begin{smallmatrix}
    A_d\Delta +\Delta A_d +\varepsilon  \Delta^2& \Delta B_2  & \Delta B_7\\
    E_2C_1\Delta & 0 & 0
  \end{smallmatrix}
  \right]
    \left[
  \begin{smallmatrix}
    x(0)\\
    u_2(0)\\
    u_7(0)
  \end{smallmatrix}
  \right] 
  +
  \left[
  \begin{smallmatrix}
    H_2C_1 & M_2
  \end{smallmatrix}
  \right]
  \left[
  \begin{smallmatrix}
    A_dB_2 & B_2 \\
    E_2C_1B_2 &0
  \end{smallmatrix}
  \right]
  \left[
  \begin{smallmatrix}
    u_2(0)\\
    u_2(1)
  \end{smallmatrix}
  \right]
\end{align*}
Recall that the agent $1$ is able to identify the misbehaving agent
$2$ if, independent of $u_2^1$ and $u_7^1$, there exists a threshold
$T$ such that $\|\hat r_2(2)\|_{\infty}\ge T$, and $\|\hat
r_3(2)\|_{\infty}< T$. The behavior of $\|\hat r_2(2)\|_{\infty}$ and
$\|\hat r_3(2)\|_{\infty}$ as a function of $\varepsilon$ is in
Fig. \ref{maxmin}. Note that for $\varepsilon=\varepsilon^*=0.026$ we
have $\|\hat r_2(2)\|_{\infty}=\|\hat r_3(2)\|_{\infty}=0.07$.
For instance, if $\varepsilon=0.01$, then it can be verified that
$\|\hat r_2(2)\|_{\infty} > 0.1$, and $\|\hat r_3(2)\|_{\infty} <
0.05$. It follows that a threshold $T=0.1$ allows the agent $1$ to
identify the misbehaving agent $2$. On the other hand, if $\varepsilon
= 0.03$, then $\|\hat r_2(2)\|_{\infty} \ge 0.01$, and $\|\hat
r_3(2)\|_{\infty} \le 0.12$, so that the misbehaving agent $2$ may
remain unidentified. Indeed, if $x(0)=\left[\begin{smallmatrix}1 & 1 &
    1 & -1 & -1 & -1 & -1\end{smallmatrix}\right]$,
$u_2^1=u_7^1=\left[\begin{smallmatrix}0.1 &
    0.1\end{smallmatrix}\right]$, then $\|\hat r_2(2)\|_{\infty}=0.01$
and $\|\hat r_3(2)\|_{\infty} = 0.12$, so that the agent $3$ is
recognized as misbehaving instead of the agent $2$.


As a final remark, note that the larger the consensus network, the
more convenient the proposed approximation procedure becomes. For
instance, consider the network presented in \cite{EB-MA:07}, and here
reported in Fig. \ref{big_clusters}. Such a clustered interconnection
structure, in which the edges connecting different clusters have a
small weight, may be preferable in many applications because much
simpler and efficient protocols can be implemented within each
cluster. Assume the presence of a misbehaving agent in each cluster,
and consider the residuals computed after $5$ steps of the consensus
algorithm. Let $\varepsilon$ be the weight of the edges connecting
different clusters. Fig. \ref{maxmin_clusters} shows, as a function of
$\varepsilon$, the smallest magnitude of the residual associated with
a misbehaving agent (dashed line) versus the largest magnitude of the
residual associated with a well-behaving agent (solid line). If
$\varepsilon$ is sufficiently small, then our local identification
method allows each well-behaving agent to promptly detect and identify
the misbehaving agents belonging to the same group, and hence to
restore the functionality of the network. For instance, if
$\varepsilon \le 0.01$, then, following Theorem
\ref{local_identification}, if the misbehaving input take value in
$\{u: 0.1 \le |u| \le 3\}$, then a misbehaving agent is correctly
detected and identified by a well-behaving agent.


\begin{figure}
  \centering \subfigure[]{
    \includegraphics[width=.4\columnwidth]{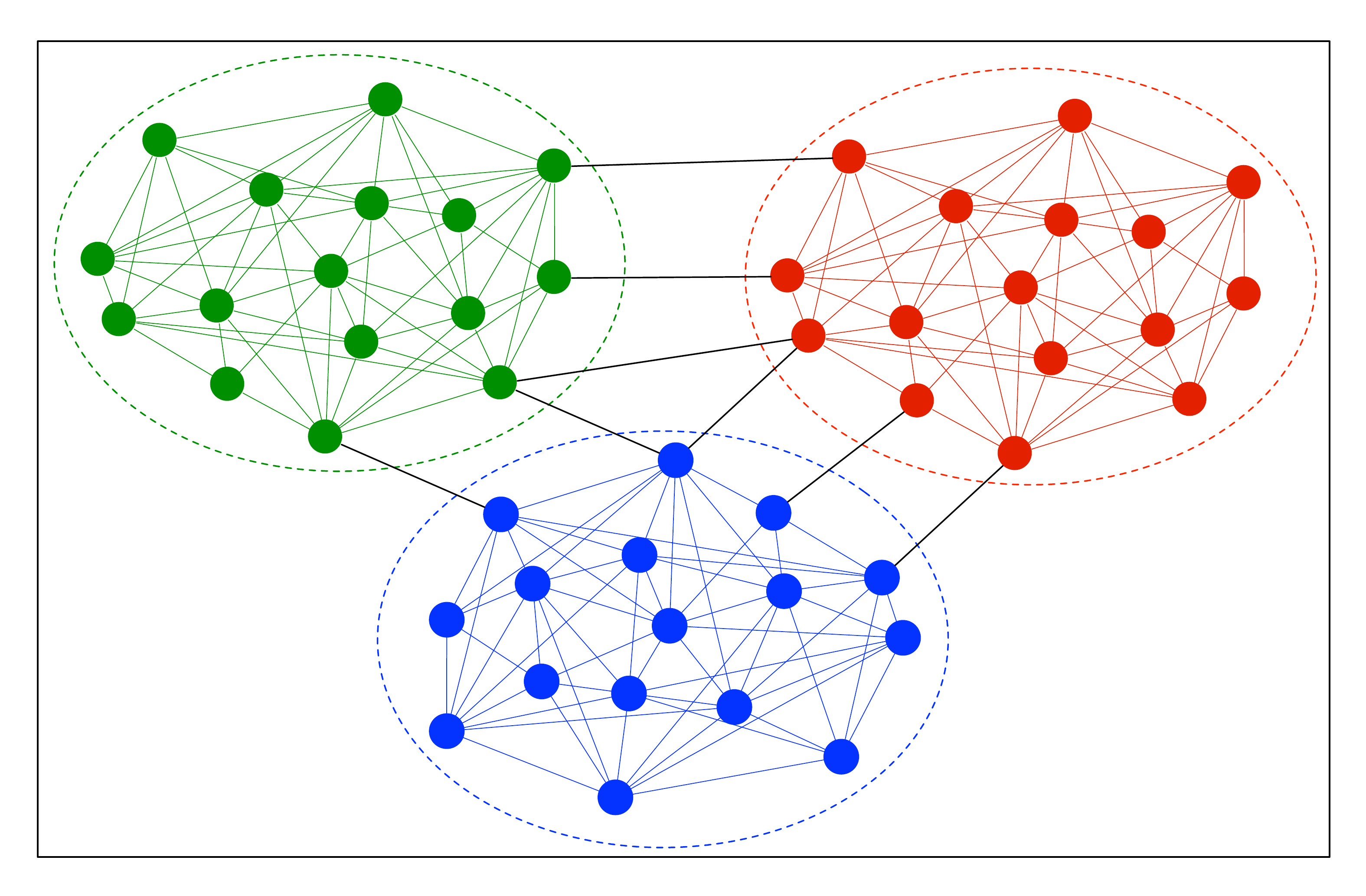}
    \label{big_clusters}
  } \subfigure[]{
    \includegraphics[width=.4\columnwidth]{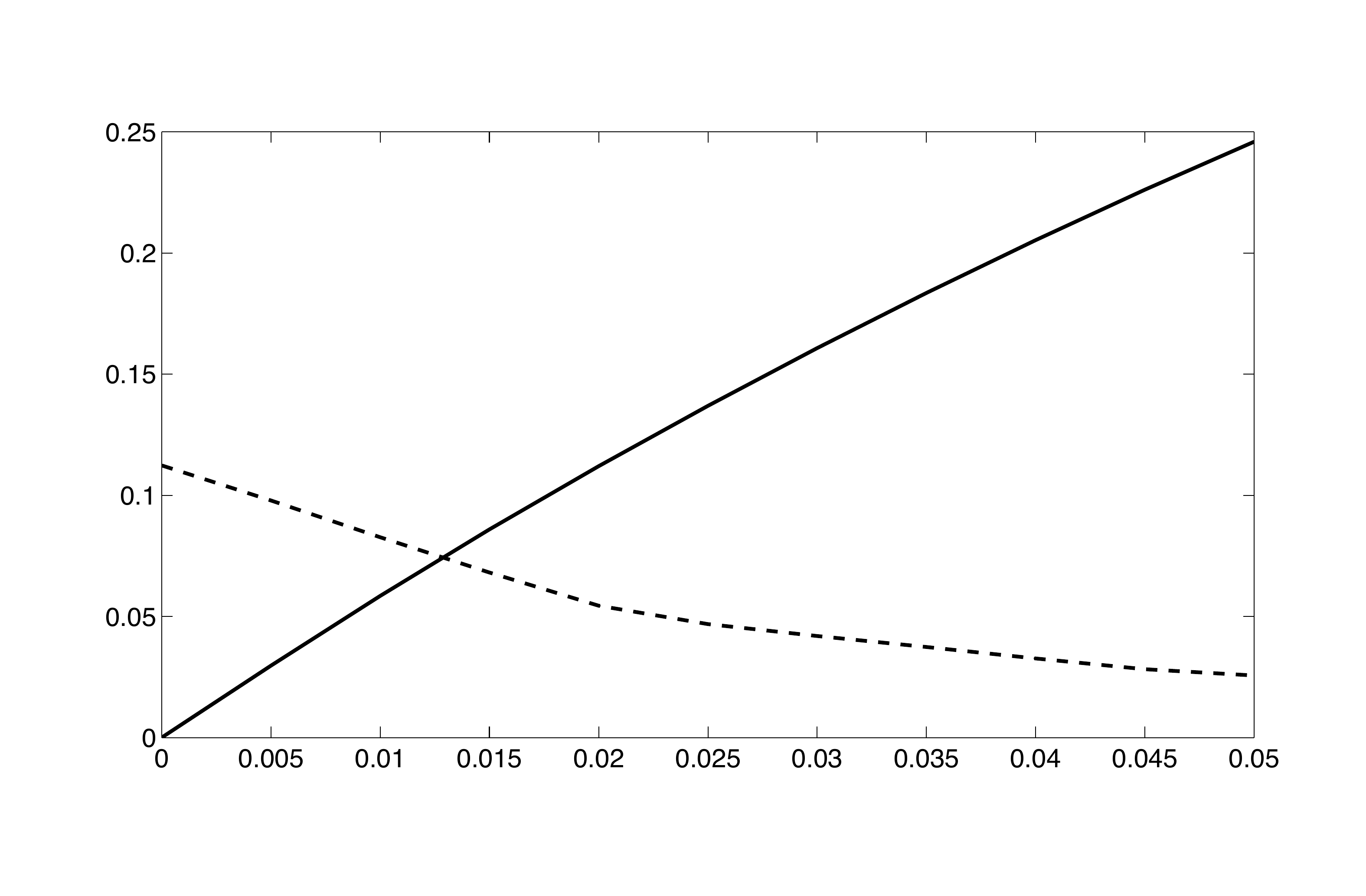}
    \label{maxmin_clusters}
  }
 \caption[Optional caption for list of figures]{In
    Fig. \ref{big_clusters} a consensus network partitioned into $3$
    areas. Each agent identifies the neighboring misbehaving agents
    by knowing only the topology of the subnetwork it belongs to. In
    Fig. \ref{maxmin_clusters} the smallest magnitude of the residual
    associated with a misbehaving agent (dashed line) and the largest
    magnitude of the residual associated with a well-behaving agent
    (solid line) are plotted as a function of $\varepsilon$. If
    $\varepsilon$ is sufficiently small, then local detection and
    identification is possible.}
\end{figure}

\section{Conclusion}\label{conclusions}
The problem of distributed reliable computation in networks with
misbehaving nodes is considered, and its relationship with the fault
detection and isolation problem for linear systems is discussed. The
resilience of linear consensus networks to external attacks is
characterized through some properties of the underlying communication
graph, as well as from a system-theoretic perspective. In almost any
linear consensus network, the misbehaving components can be correctly
detected and identified, as long as the connectivity of the
communication graph is sufficiently high. Precisely, for a linear
consensus network to be resilient to $k$ concurrent faults, the
connectivity of the communication graph needs to be $2k+1$, if
Byzantine failures are allowed, and $k+1$, otherwise. Finally, for the
faulty agents case, good performance can be obtained even if the
agents do not know the entire network topology, and they are
subject to memory or computation constraints.

Interesting aspects requiring further investigation include a
characterization of the gain between the inputs of a set of
misbehaving agents and the observations of an agent $j$. Depending on
the magnitude of such gain, some undetectable behaviors may not be
feasible for a set of misbehaving agents. The resilience properties of
specific consensus protocols, e.g., those resulting from an
optimization process, should also be studied. Finally, the clustering
of a large network into smaller parts is crucial for the performance
of the proposed local identification procedure, and it requires
additional research.


\bibliographystyle{IEEEtran}

\end{document}